\documentclass[reqno]{amsart} 
\usepackage{enumerate,amsmath,amssymb, color} 
\usepackage{hyperref}

\theoremstyle{plain} 
 
\newtheorem{thm}{Theorem}[section] 
\newtheorem{theorem}[thm]{Theorem} 
 
\newtheorem{corollary}[thm]{Corollary} 
 
\newtheorem{lemma}[thm]{Lemma} 
 
\newtheorem{proposition}[thm]{Proposition} 
 
\theoremstyle{remark}

\theoremstyle{definition} 
 
\newtheorem{definition}[thm]{Definition} 

\def\de{{\delta}}

\def\Om{{\Omega}}
\def\la{{\lambda}}

\def\Si{{\Sigma}}

\def\epsilon{{\varepsilon}}
\def\ep{{\varepsilon}}

\def\phi{{\varphi}}

\DeclareMathAlphabet{\doba}{U}{msb}{m}{n} 
\gdef\mC{\doba{C}}

\gdef\mN{\doba{N}}

\gdef\mR{\doba{R}}

\def\vol{{\mathop{\rm vol}}}

\let\scal\Scal

\let\<\langle 
\let\>\rangle

\newcommand{\definedas}{\mathrel{\raise.095ex\hbox{\rm :}\mkern-5.2mu=}}

\newcounter{mnotecount}[section]

\def\dist{\mathrm{d}}
\def\supp{\mathrm{supp}}
\def\diff{d}

\begin{document}

\title{About the mass of certain second order elliptic operators}
\author{Andreas Hermann}
\address{Andreas Hermann, LMPT \\
  Universit\'e de Tours,\\
Parc de Grandmont \\
   37200 Tours \\
   France} 
\email{andreas.hermann@lmpt.univ-tours.fr} 

\author{Emmanuel Humbert}
\address{Emmanuel Humbert, LMPT \\
Universit\'e de Tours,\\
Parc de Grandmont \\
   37200 Tours \\
   France} 
\email{emmanuel.humbert@lmpt.univ-tours.fr}

\begin{abstract}
Let $(M,g)$ be a closed Riemannian manifold of 
dimension $n \geq 3$ and let 
$f\in C^{\infty}(M)$, such that the 
operator $P_f:= \Delta_g+f$ is positive. If $g$ is flat near some point $p$ and $f$ vanishes around $p$, we can define the mass of $P_f$ as the constant term in the expansion of the Green function of $P_f$ at $p$. In this paper, we establish many results on the mass of such operators. In particular, if $f:= \frac{n-2}{4(n-1)} \scal_g$, i.e. if $P_f$ is the Yamabe operator, we show the following result: assume that there exists a closed simply connected non-spin manifold $M$ such that the mass is non-negative for every metric $g$ as above on $M$, then the mass is non-negative for every such metric on every closed manifold of the same dimension as $M$.   
\end{abstract}

\maketitle

\tableofcontents


\section{Introduction} 

Let $(M,g)$ be a closed Riemannian manifold of dimension $n \geq 3$, let $p\in M$ and assume that 
$g$ is flat on an  open neighborhood $U$ of $p$. 
Let $f\in C^{\infty}(M)$ such that $f\equiv 0$ 
on $U$.  
Then, a {\em Green function} of 
$P_f:=\Delta_g+f$ at $p$ is a function 
$G_f \in L^1(M) \cap C^{\infty}( M \setminus\{p\})$ 
such that in the sense of distributions 
\begin{equation} 
 P_f G_f = \delta_p
\end{equation}
where $\delta_p$ is the Dirac distribution at $p$. 
It is well known that
 
\begin{proposition} 
Assume that all eigenvalues of 
the operator $P_f$ are positive. 
Then, there exists a unique Green function $G_f$ 
for $P_f$ at $p$. 
Moreover, $G_f$ is strictly positive on 
$M\setminus\{p\}$ and 
has the following expansion at $p$:
\begin{equation} 
G_f = \frac{1}{(n-2)\omega_{n-1} r^{n-2}} + m_f + o(1)
\end{equation}
where $r:=\dist_g(p,\cdot)$ is the distance function to 
$p$, where $\omega_{n-1}$ is the volume of the 
standard $(n-1)$-sphere and where $m_f$ is a number 
called the {\em mass of $P_f$ at the point $p$}. 
\end{proposition}

Considering the importance of this proposition for 
this paper, we give the proof in Section 
\ref{prelimin}.
These objects play a crucial role in many problems 
of geometric analysis in which blowing-up 
sequences of functions behave like Green function. 
The most famous one is maybe the Yamabe problem 
which consists in finding a metric with constant 
scalar curvature in a given conformal class. 
After Yamabe, Trudinger and Aubin had found 
a solution to this problem in some special cases, 
the remaining cases were solved by Schoen
in 1984 with a test function argument in which he 
used the Green function of the 
{\em conformal Laplacian} or {\em Yamabe operator} 
\begin{displaymath}
L_g:= \Delta_g + \frac{n-2}{4(n-1)} \scal_g.
\end{displaymath}
We give more information on the operator $L_g$ in Paragraph \ref{yamabeop}. With the notation above, $L_g= P_{\frac{n-2}{4(n-1)} \scal_g}$. Schoen could show that the positivity of  the number $m_{\frac{n-2}{4(n-1)} \scal_g}$ allows to solve the Yamabe problem. 
To prove this last step, he showed that $m_{\frac{n-2}{4(n-1)} \scal_g}$ can be interpreted as the ADM mass of an asymptotically flat manifold, which is regarded as the energy of an isolated system in general relativity and which can be proved to be positive in this context. 
Even if this interpretation is really specific to $m_{\frac{n-2}{4(n-1)} \scal_g}$, the number $m_f$ for a more general $f$ is now called {\em mass of the operator $P_f$}. For more information on the Yamabe problem, we refer the reader for instance to \cite{lee.parker:87}. 

\noindent At a first glance, we could think from the definition that the mass $m_f$ only depends on the local geometry around $p$. Unfortunately, this is not true which makes its study very difficult. 
In particular, the question of whether $m_{\frac{n-2}{4(n-1)} \scal_g} \geq 0$ with equality if and only $(M,g)$ is conformally equivalent to the standard sphere is still open in full generality. It is proven only in some particular cases, including the context of Yamabe problem (i.e. when $(M,g)$ is locally conformally flat, see \cite{schoen.yau:88}) and the case of spin manifolds, solved by Witten in \cite{witten:81}.  \\

\noindent The first result of this paper is Theorem \ref{maintheorem} in which we show that $- m_f$ can be expressed as the minimum of a functional. Note that Hebey and Vaugon \cite{hebey.vaugon:91} have already proved a variational characterization of the mass $m_{\frac{n-2}{4(n-1)} \scal_g}$ but their approach is different giving rise to different applications.  
We then exhibit four short applications of Theorem \ref{maintheorem}: 
\begin{itemize} 
\item We first give an alternative proof of the positive mass theorem on spin manifolds. This proof is not simpler than the one of Ammann-Humbert \cite{ammann.humbert:05} but has the advantage to enlighten the ingredients which make the proof work. 
\item We prove in a very simple way a generalization of a result of Beig and O'Murchadha who proved in \cite{beig.omurchada:97} that near a metric of zero Yamabe constant, the mass $ m_{\frac{n-2}{4(n-1)} \scal_g}$ is arbitrarily large. 
\item We prove that on every manifold, we can find many non-negative functions~$f$ for which $m_f$ 
is negative.
\item We prove that the positivity of   
$m_{\frac{n-2}{4(n-1)} \scal_g}$ is preserved by surgery (see Section \ref{section_surgery_positivity} for a precise statement).  \end{itemize}
 
\noindent These facts could also be proven directly but  Theorem \ref{maintheorem} is nevertheless interesting for many reasons: 
\begin{itemize}
\item The variational characterization is really easy to manipulate and helps a lot to simplify the proofs. For instance, the mass-to-infinity Theorem \ref{massinfty} becomes almost obvious with this approach.
\item Theorem \ref{maintheorem} makes it easy to have a good intuition without any computation of what is true or not, as can be seen for example in Section 
\ref{section_surgery_positivity} about the preservation of the positivity of mass by surgery. 
\item Theorem \ref{maintheorem} clarifies the situation a lot: this is particularly true for the proof of the positive mass theorem on spin manifolds (see Section \ref{section_pmt_spin}).
\end{itemize}

After these applications we prove that also the 
negativity of $m_{\frac{n-2}{4(n-1)}\scal_g}$ is 
preserved by surgery (see Section 
\ref{section_surgery_preservation} for a precise 
statement). 
The proof is more difficult than the proof of the 
preservation of the positivity of 
$m_{\frac{n-2}{4(n-1)}\scal_g}$ and uses 
Theorem \ref{maintheorem} together with some 
techniques developed in the article 
\cite{ammann.dahl.humbert:13}. \\

\noindent As explained above, the question of whether $m_{\frac{n-2}{4(n-1)} \scal_g} \geq 0$ with equality if and only if $(M,g)$ is conformally equivalent to the standard sphere is still open. It is known as the {\em positive mass conjecture (weak version)} and is a particular case of the standard {\em positive mass conjecture} which says that the ADM mass of an asymptotically flat manifold with non-negative and integrable scalar curvature must be non-negative and vanishes if and only if the manifold is $\mR^n$ equipped with the flat metric. It turns out that both versions of the positive  mass conjecture are actually equivalent: see Proposition 4.1 in \cite{schoen:89} or Section 5 in \cite{lohkamp:99} (this could also be proved using Theorem \ref{maintheorem} but the proof is not really simpler and not instructive so we omit it in this paper). The positive mass conjecture is proved when $n \leq 7$ by Schoen and Yau \cite{schoen.yau:79} or when $(M,g)$ is spin by Witten \cite{witten:81}. 
More recently Lohkamp has announced a complete proof 
in \cite{lohkamp:06}. 
Note that the conjecture has been proved by Schoen and Yau \cite{schoen.yau:88} under the assumption that the manifold is conformally flat leading to the complete solution of the Yamabe problem.  \\

\noindent  Now, let $M$ be a closed manifold. We say that PMT (for Positive Mass Theorem) is true on $M$ if for every point $p \in M$ and for every metric $g$ on $M$ which is flat around $p$ and for which $L_g$ is a positive operator we have $m_{\frac{n-2}{4(n-1)} \scal_g} \geq 0$. 
Using that the negativity of 
$m_{\frac{n-2}{4(n-1)}\scal_g}$ is preserved by 
surgery we obtain the second main result of this 
paper which is the following:

\begin{theorem} 
\label{main2}
 Assume that PMT is true on a closed simply connected non-spin manifold of dimension $n\geq 5$, then PMT is true on all closed manifolds of dimension $n$.  
\end{theorem}

\noindent Note that using for instance Proposition 4.1 in \cite{schoen:89} or Section 5 in \cite{lohkamp:99} 
one can conclude from the assumption of this theorem 
that every asymptotically flat Riemannian manifold 
of dimension $n$ with non-negative and integrable 
scalar curvature has non-negative ADM mass. \\


\noindent This theorem should help a lot to prove the positive mass conjecture. Indeed, it reduces the problem to finding a non-spin simply connected manifold $M$ on which PMT is true. For instance, $\mC P^{2m}$ or $\mC P^{2m}\times S^k$ with $k\geq 2$  could be a good candidate to provide such an example by using its particular structure. We did not succeed until now but let us explain how some structures could help a lot to prove that PMT is true on a manifold. First, it is not difficult to find a simply connected manifold for which PMT is true: it suffices to choose a manifold which is spin (the sphere for instance). But we can also easily construct  a non-spin manifold for which PMT is true (unfortunately, it is not simply connected): 

\begin{proposition}
 Let $n \geq 5$, $n\equiv 1 \bmod 4$. Then, the projective space $\mR P^n$ satisfies PMT. 
\end{proposition}

\noindent The proof of this proposition is really simple and is given is Section \ref{section_pmt_non_spin}. \\

\noindent The paper is organized as follows: 
\begin{itemize}
 \item In Section \ref{prelimin}, we give some general preliminaries which will be used in the whole text;
\item In Section \ref{variational}, we give the statement of Theorem \ref{maintheorem} whose goal is to establish the variational characterization of the mass;
\item Sections \ref{section_proof_main_thm} and \ref{proof'} are devoted to the proof of Theorem \ref{maintheorem};
\item In Section \ref{applications}, we give several applications of Theorem \ref{maintheorem};
\item In Section \ref{section_surgery_preservation}, we establish a surgery formula for the mass which will be the main ingredient in the proof of Theorem \ref{main2}; 
\item In Section \ref{section_pmt_non_spin}, we show how the results of Section \ref{section_surgery_preservation} can be applied to prove Theorem \ref{main2}.
\end{itemize}

\noindent {\bf Acknowledgements: }
The authors would like to thank Bernd Ammann and 
Mattias Dahl 
for many enlightening discussions on the subject. 
A. Hermann is supported by the DFG research grant 
HE 6908/1-1. 
E. Humbert is partially supported by ANR-10-BLAN 0105 and by ANR-12-BS01-012-01.

\section{Preliminaries} 
\label{prelimin}

In these sections, we introduce all the objects and  the notation which will be needed in the paper and we give some additional information on the context of the problem. 

\subsection{Notation} 

All manifolds are assumed to be connected and 
without boundary unless otherwise stated. 
We denote by $\xi^n$ the Euclidean metric on $\mR^n$ 
and by $\sigma^n$ the standard metric of constant 
sectional curvature $1$ on $S^n$. 
For any Riemannian manifold $(M,g)$ and for $p\in M$ 
and $r>0$ we denote by $B(p,r)$ or by 
$B^g(p,r)$ the open ball of radius $r$ centered 
at $p$. 
For a subset $N$ of $M$ we denote by 
$\vol(N)$ or $\vol^g(N)$ 
the volume of $N$ with respect to $g$ 
and by $\dist_g(x,N)$ the distance of $x$ to $N$. 
The scalar curvature of any Riemannian metric $g$ 
will be denoted by $\scal_g$. 
We will use the abbreviation
\begin{displaymath}
\int_{M\setminus\{p\}}
:=\lim_{\varepsilon\to0}
\int_{M\setminus B(p,\varepsilon)}.
\end{displaymath}
For any Riemannian manifold $(M,g)$ and for any 
$q\in[1,\infty]$ 
we denote by $L^q(M)$ the space of all measurable 
functions on $M$ with finite $L^q$-norm. 
The Sobolev space $H^{1,2}(M)$ is the space of all 
functions in $L^2(M)$ whose distributional 
derivative exists and is in $L^2(M)$.

\subsection{A cut-off formula}
We state a formula which is used several times 
in the article (see also Appendix A.3 in 
\cite{ammann.dahl.humbert:13}). 
Let $u$ and $\chi$ be smooth functions on a 
Riemannian manifold $(M,g)$ and assume that 
$\chi$ has compact support. 
Then we have 
\begin{align}
\nonumber 
\int_{M} |\diff(\chi u)|^2\,dv^g &=
\int_{M} |u \diff\chi+\chi \diff u|^2\,dv^g\\
\nonumber 
&=\int_{M}
(u^2|\diff\chi|^2
+g( \chi^2 \diff u,\diff u)
+g(2u\chi \diff\chi,\diff u))\,dv^g\\
\nonumber 
&=\int_{M}
(u^2|\diff\chi|^2
+g(\chi^2 \diff u, \diff u)
+g(u\diff(\chi^2),\diff u))\,dv^g\\
\nonumber 
&=\int_{M}
(u^2|\diff\chi|^2
+g(\diff(\chi^2u),\diff u))\,dv^g\\
\label{cut_off_formula}
&=\int_{M}
(u^2|\diff\chi|^2+\chi^2 u\Delta_g u)\,dv^g.
\end{align}

\subsection{Properties of the Green function} 

Let $(M,g)$ be a closed Riemannian manifold of 
dimension $n\geq 3$. 
Let $f\in C^{\infty}(M)$ and assume that the 
operator $P_f:=\Delta_g+f$ acting on 
$C^{\infty}(M)$ has only positive eigenvalues. 
Fix $p\in M$. 
A function 
$G_f\in L^1(M) \cap C^{\infty}( M\setminus\{p\})$ 
is called a Green function for $P_f$ at $p$ 
if for all $u\in C^{\infty}(M)$ we have 
\begin{displaymath}
\int_{M\setminus\{p\}} G_f P_f u\, dv^g = u(p). 
\end{displaymath}
In our article we use the following properties of 
the Green function which are well known. 

\begin{proposition}
\label{green_function}
Assume that $P_f$ is a positive operator. 
Then the following holds. 
\begin{enumerate}[1.]
\item At every point $p\in M$ there exists a unique 
Green function $G_f$ for $P_f$. 
Moreover $G_f$ is strictly positive on 
$M\setminus\{p\}$.
\item Let $p\in M$ and assume that 
there exists an open neighborhood $U$ of $p$ such 
that $g$ is flat on $U$ and $f\equiv 0$ on $U$. 
Then the function $G_f$ has the following 
expansion as $x\to p$
\begin{equation}
\label{expansion} 
G_f(x)=\frac{1}{(n-2)\omega_{n-1}r^{n-2}}
+m_f+o(1),
\end{equation}
where $r:=\dist_g(p,\cdot)$ is the distance function 
to $p$, $\omega_{n-1}$ is the volume of the standard 
$(n-1)$-sphere and $m_f$ is a real number 
called the mass of $P_f$ at $p$. 
\end{enumerate} 
\end{proposition}

\begin{proof}
1.: The proof is classical and we omit it here.\\
2.: Let $\eta$ and $F_{\eta}$ be as in Section 
\ref{variational}. 
Since $P_f$ has only positive eigenvalues, 
$P_f$ is 
invertible on $C^{\infty}(M)$. 
Let $v:= P_f^{-1}(F_{\eta})$. 
The function $G_f := \eta r^{2-n} - v$ is smooth 
on $M \setminus \{ p \}$, is in $L^1(M)$ and 
satisfies $P_f G_f= 0$ on $M \setminus \{ p\}$. 
Moreover, near $p$, 
\begin{displaymath}
G_f(x) = \frac{1}{(n-2) \omega_{n-1} r^{n-2}}
+v(x)
\end{displaymath}
where $P_f v =  \Delta_g v= 0$. 
Since the manifold is flat around $p$ and thus 
locally isometric to a neighborhood of $0$ 
in $\mR^n$ and since the Green function for the 
Laplacian on $\mR^n$ at $0$ is 
$\frac{1}{(n-2) \omega_{n-1} r^{n-2}}$, we get 
that $P_f v = \delta_p$ and thus $G_f$ is a 
Green function for $P_f$. 
This proves the existence.

If now $G$ and $G'$ are Green functions 
for $P_f$ then $P_f(G-G') = 0$ in the sense of 
distributions. 
By standard regularity theorems, $G-G'$ is smooth and 
hence, by invertibility of $P_f$ we obtain $G=G'$. 
\end{proof}

\subsection{The Yamabe operator} 
\label{yamabeop}
Let $(M,g)$ be a closed Riemannian manifold of 
dimension $n \geq 3$. 
We define $f:=\frac{n-2}{4(n-1)}\scal_g$ and 
denote the operator $P_f$ by 
\begin{displaymath}
L_g:= \Delta_g + \frac{n-2}{4(n-1)} \scal_g.
\end{displaymath}
This operator is called the 
{\em conformal Laplacian} or {\em  Yamabe operator}. 
If the metric $g$ is flat on an open neighborhood 
of a point $p\in M$, we will denote the mass 
of $L_g$ at $p$ by $m(M,g)$. 
There are several reasons why this operator is very 
important. 
First it played a crucial role in the solution of 
the Yamabe problem, which is a famous problem 
in conformal geometry. 
For more information on the subject, the reader may 
refer to \cite{aubin:98, hebey:97, lee.parker:87}. 
Furthermore the mass of the operator $L_g$ can be 
interpreted as the ADM mass of an asymptotically 
flat Riemannian manifold, which is an important 
quantity measuring the total energy of an isolated 
gravitational system in general relativity 
(see \cite{schoen:84}). 

In this article we will use several properties of 
the operator $L_g$. 
First it transforms nicely under 
conformal changes of the metric. 
Namely, if $g'= u^{\frac{4}{n-2}} g$ are two 
conformally related metrics, where $u$ is 
a smooth positive function on $M$, then 
for all $\varphi\in C^{\infty}(M)$ we have 
\begin{equation}
\label{L_g_conform}
L_{g'}(u^{-1}\varphi)=u^{-\frac{n+2}{n-2}}
L_g(\varphi)
\end{equation}
(see e.\,g.\,\cite{lee.parker:87}, p.\,43). 
Using this formula with $\varphi=u$ we obtain 
the equation 
\begin{equation} 
\label{scalgg'}
L_{g}(u)
= \frac{n-2}{4(n-1)}\scal_{g'} u^{\frac{n+2}{n-2}},
\end{equation}
which gives a relation between the scalar curvatures 
of $g$ and $g'$. 
Next we define 
\begin{displaymath}
Y(M,g) := 
\inf\Big\{
\frac{ \int_M u L_g u\, dv^{g}}
{(\int_M |u|^p\, dv^g)^{2/p}}
\Big|\,
u\in C^{\infty}(M),\,u\not\equiv 0
\Big\},
\end{displaymath}
where $p:=\frac{2n}{n-2}$. 
This number is a conformal invariant called the 
{\em Yamabe constant} of $(M,g)$. 
The operator $L_g$ is positive (i.e. has only 
positive eigenvalues) if and only if $Y(M,g)$ is 
positive. 

If $g'=u^{\frac{4}{n-2}}g$ and 
if $g$ and $g'$ are both flat in an 
open neighborhood of a point $p\in M$ and if 
$G$ and $G'$ denote the Green functions of 
$L_g$ and $L_{g'}$ respectively, we have for 
all $x\in M\setminus\{p\}$
\begin{displaymath}
G'(x)=u(p)^{-1}u(x)^{-1}G(x)
\end{displaymath}
(see e.g.\cite{lee.parker:87}, p.\,63). 
If we write down the expansions of $G$ and $G'$ 
given by Proposition \ref{green_function} and use 
that $u$ is constant on an open neighborhood of $p$, 
it follows that $m(M,g)$ and $m(M,g')$ have the 
same sign 
(see also \cite{schoen.yau:88} or \cite{hebey:97}, p.\,277).
\section{A variational characterization of the mass}  
\label{variational} 

We keep the same notation as above and fix a function 
$f$ such that the operator $P_f$ is positive. 
Then the Green function $G_f$ of $P_f$ at $p$ 
and the associated mass $m_f$ are well defined.  
Let $\delta>0$ such that the ball $B(p,\delta)$ 
around $p$ of radius $\delta$ is contained in $U$ 
and let $\eta$ be a smooth function on $M$ 
such that $\eta\equiv\frac{1}{(n-2)\omega_{n-1}}$ 
on $B(p,\delta)$ and $\supp(\eta)\subset U$, 
where $\omega_{n-1}$ denotes the volume of $S^{n-1}$ 
with the standard metric. The function $F_{\eta}$: 
$M\to\mR$ defined by 
\begin{displaymath}
F_{\eta}(x)=
\left\{\begin{array}{ll}
\Delta_g(\eta r^{2-n})(x),&x\neq p\\
0, & x=p
\end{array}\right.
\end{displaymath}
is smooth on $M$. 
For every $u\in C^{\infty}(M)$ we define 
\begin{displaymath}
I_f(u):=\int_{M\setminus\{p\}}
(\eta r^{2-n}+u)P_f(\eta r^{2-n}+u)\,dv^g,
\end{displaymath}
and 
\begin{displaymath}
J_f(u):=\int_{M\setminus\{p\}}
\eta r^{2-n}F_{\eta}\,dv^g
+2\int_{M}
uF_{\eta}\,dv^g
+\int_{M}uP_f u\,dv^g.
\end{displaymath}
We also define 
\begin{align*}
\nu&:=\inf\{I_f(u)|\,
u \in C^{\infty}(M),\, u(p)= 0\},\\
\mu&:=\inf\{J_f(u)|\,u\in C^{\infty}(M)\}.
\end{align*}
Let us remark the following fact: if $\eta'$ is 
another smooth function with the same properties 
as $\eta$, one can construct in a similar way:
\begin{displaymath}
I'_f (u) :=\int_{M\setminus\{p\}}
(\eta' r^{2-n}+u)P_f(\eta' r^{2-n}+u)\,dv^g.
\end{displaymath} 
Note that, for all $u$,
\begin{displaymath}
I_f  ( u) = I'_f(u - \eta' r^{2-n} + \eta r^{2-n})
\end{displaymath}  
and that 
$u-\eta'r^{2-n}+\eta r^{2-n}$ has a smooth 
extension to all of $M$. 
As a consequence, the number $\nu$ does not depend 
on the choice of $\eta$.

The following theorem is the main result 
of this article. 

\begin{theorem} \label{maintheorem} 
We have $\nu= \mu= -m_f= J_f(G_f-\eta r^{2-n})$.  
\end{theorem}

\noindent The proof is obtained in several steps and is done 
in Section \ref{section_proof_main_thm}.

\section{Proof of Theorem \ref{maintheorem}} 
\label{section_proof_main_thm}

The proof of Theorem \ref{maintheorem} proceeds in several steps. The general idea is to show that $\nu$ and $\mu$ are equal and that $\mu$  is attained by exactly one smooth function $u$ which is such that 
\begin{displaymath}
G_f = \eta r^{2-n} + u.
\end{displaymath} 
These facts will be established in the following 
lemmas. 
First we relate the functionals $I_f$ and $J_f$. 

\begin{lemma}
\label{lemma_Ig2}
For all $u\in C^{\infty}(M)$ we have 
\begin{displaymath}
I_f (u)=  J_f(u) + u(p).
\end{displaymath}
\end{lemma}

\begin{proof}
Using that $f\equiv 0$ on $\supp(\eta)$ 
we calculate 
\begin{align*}
I_f(u)&=
\int_{M\setminus\{p\}}
\eta r^{2-n}\Delta_g(\eta r^{2-n})\,dv^g
+\int_{M\setminus\{p\}}
u\Delta_g(\eta r^{2-n})\,dv^g\\
&{}+\int_{M\setminus\{p\}}
\eta r^{2-n}\Delta_gu\,dv^g
+\int_{M}uP_fu\,dv^g.
\end{align*}
Let $\varepsilon>0$ and let $\nu$ be the 
unit normal vector field on 
$\partial B(p,\varepsilon)$ pointing into 
$M\setminus B(p,\varepsilon)$. 
Integrating by parts,  we have 
\begin{align*}
&\int_{M\setminus B(p,\varepsilon)}
\eta r^{2-n}\Delta_gu\,dv^g
-\int_{M\setminus B(p,\varepsilon)}
u\Delta_g(\eta r^{2-n})\,dv^g\\
&=\int_{\partial B(p,\varepsilon)}
\eta r^{2-n}\partial_{\nu}u\,ds^g
-\int_{\partial B(p,\varepsilon)}
u\partial_{\nu}(\eta r^{2-n})\,ds^g.
\end{align*}
As $\varepsilon\to 0$, the first term on the 
right hand side tends to $0$ and the second 
integral on the right hand side tends to $-u(p)$. 
The assertion follows. 
\end{proof}

\begin{lemma} \label{existu}
We have $\mu>-\infty$ and $\nu>-\infty$. 
Furthermore there exists a unique function 
$u\in C^{\infty}(M)$ such that $\mu=J_f(u)$. 
\end{lemma}

\begin{proof}
Assume that there exists a sequence 
$(u_k)_{k\in\mN}$ in $C^{\infty}(M)$ such that 
$J_f(u_k)\to-\infty$ as $k\to\infty$. 
Since $P_f$ is a positive operator, there exists 
$A>0$ such that 
for all $k\in\mN$ we have 
\begin{displaymath}
\int_M u_kP_fu_k\,dv^g\geq A\|u_k\|_{L^2(M)}^2\geq 0.
\end{displaymath} 
From our assumption and the definition 
of $J_f$ it follows 
that $\int_{M}u_kF_{\eta}\,dv^g\to-\infty$ 
as $k\to\infty$. 
On the other hand 
with H\"older's inequality we have 
for all $k\in\mN$ 
\begin{displaymath}
\Big|\int_{M}u_kF_{\eta}\,dv^g\Big|
\leq
\|F_{\eta}\|_{L^2(M)}\,
\|u_k\|_{L^2(M)}.
\end{displaymath}
Thus we have $\|u_k\|_{L^2(M)}\to\infty$ 
as $k\to\infty$ and thus 
\begin{displaymath}
J_f(u_k)\geq
\int_{M\setminus\{p\}}
\eta r^{2-n}F_{\eta}\,dv^g
-2\|F_{\eta}\|_{L^2(M)}\,\|u_k\|_{L^2(M)}
+A\|u_k\|_{L^2(M)}^2\to\infty
\end{displaymath}
as $k\to\infty$, which is a contradiction. 
Thus we have $\mu>-\infty$. 
Next assume that there exists a sequence 
$(u_k)_{k\in\mN}$ in $C^{\infty}(M)$ such that 
for every $k\in\mN$ we have $u_k(p)=0$ and 
$I_f(u_k)\to-\infty$ as $k\to\infty$. 
By Lemma \ref{lemma_Ig2} we conclude 
$J_f(u_k)\to-\infty$ which is a contradiction. 
Thus we have $\nu>-\infty$. 

Let $(u_k)_{k\in\mN}$ be a sequence in 
$C^{\infty}(M)$ such that 
$J_f(u_k)\to\mu$ as $k\to\infty$. 
As above it follows that $(u_k)_{k\in\mN}$ 
is bounded in $L^2(M)$. 
Since for all $k\in\mN$ we have 
\begin{align*}
J_f(u_k)&=
\int_{M\setminus\{p\}}
\eta r^{2-n}F_{\eta}\,dv^g
+2\int_{M}
u_kF_{\eta}\,dv^g\\
&{}+\int_{M}|\diff u_k|^2\,dv^g
+\int_{M}fu_k^2\,dv^g,
\end{align*}
it follows that the sequence 
$(|\diff u_k|)_{k\in\mN}$ 
is bounded in $L^2(M)$ 
and thus that $(u_k)_{k\in\mN}$ 
is bounded in $H^{1,2}(M)$. 
Since $H^{1,2}(M)$ is reflexive 
there exists $u\in H^{1,2}(M)$ such that 
after passing to a subsequence we have 
$u_k\to u$ weakly in $H^{1,2}(M)$. 
Furthermore since the embeddings of $H^{1,2}(M)$ 
into $L^1(M)$ and into $L^2(M)$ are compact 
we obtain after passing again to sub-sequences 
that $u_k\to u$ strongly in $L^1(M)$ 
and in $L^2(M)$. 
For every $k\in\mN$ we have 
\begin{align*}
0&\leq\int_{M}|\diff u-\diff u_k|^2\,dv^g\\
&=\int_{M}|\diff u|^2\,dv^g
+\int_{M}|\diff u_k|^2\,dv^g
-2\int_{M}g(\diff u,\diff u_k)\,dv^g.
\end{align*}
By weak convergence in $H^{1,2}(M)$ the third 
term on the right hand side converges to 
$-2\int_{M}|\diff u|^2\,dv^g$ as $k\to\infty$. 
It follows that  
\begin{displaymath}
\int_{M}|\diff u|^2\,dv^g
\leq\liminf_{k\to\infty}
\int_{M}|\diff u_k|^2\,dv^g.
\end{displaymath}
Since the sequence $(u_k)_{k\in\mN}$ converges 
strongly to $u$ in $L^1(M)$ 
and in $L^2(M)$ we have 
\begin{displaymath}
\int_{M}fu_k^2\,dv^g
\to\int_{M}fu^2\,dv^g,\quad 
\int_{M}u_kF_{\eta}\,dv^g
\to\int_{M}uF_{\eta}\,dv^g
\end{displaymath}
as $k\to\infty$. 
It follows that 
\begin{displaymath}
J_f(u)\leq\liminf_{k\to\infty}J_f(u_k)=\mu
\end{displaymath}
and therefore $J_f(u)=\mu$. 
For every $\varphi\in C^{\infty}(M)$ we have 
\begin{displaymath}
0
=\frac{d}{dt}J_f(u+t\varphi)\big|_{t=0}
=2\int_{M}\varphi F_{\eta}\,dv^g
+2\int_{M}\varphi P_fu\,dv^g
\end{displaymath}
and therefore $P_fu=-F_{\eta}$. 
Using standard results in regularity theory 
(see e.\,g.\,\cite{gilbarg.trudinger:77}) 
we see from this equation that $u$ is smooth. 
We also see that $u$ is the unique minimizer 
of $J_f$ since $P_f$ is invertible on 
$C^{\infty}(M)$. 
\end{proof}

\begin{lemma} \label{lemma_mu=-m}
We have $\mu=-m_f$.
\end{lemma}

\begin{proof}
Define $v:=u+\eta r^{2-n}$. 
Then $G_f-v$ has a smooth extension to all of $M$ 
and on $M\setminus\{p\}$ we have 
\begin{displaymath}
P_f(G_f-v)=P_f(G_f-\eta r^{2-n}-u)
=-F_{\eta}-P_fu=0.
\end{displaymath}
Since $P_f$ is invertible on smooth functions, 
we have $v=G_f$. 
It follows that $u(p)=m_f$ and therefore 
\begin{displaymath}
\mu=J_f(u)=I_f(u)-m_f
=\int_{M\setminus\{p\}}G_fP_fG_f\,dv^g-m_f=-m_f.
\end{displaymath}
This ends the proof. 
\end{proof}

\noindent We are now able to prove the result:
\begin{lemma}
\label{lemma_mu=nu}
We have 
\begin{displaymath}
\mu=\nu.
\end{displaymath}
\end{lemma}

\noindent Together with Lemma \ref{lemma_mu=-m} 
this proves Theorem \ref{maintheorem}
\begin{proof}
In order to show "$\mu\leq\nu$" 
let $\varepsilon>0$ and let 
$u\in C^{\infty}(M)$ such that $u(p)=0$ 
and $I_f(u)\leq\nu+\varepsilon$. 
Then we have $J_f(u)=I_f(u)\leq\nu+\varepsilon$ 
and thus $\mu\leq\nu+\varepsilon$. 

In order to show "$\mu\geq\nu$" let $\varepsilon>0$ 
and let 
$u\in C^{\infty}(M)$ such that 
$J_f(u)=\mu$. 
For $s>0$ let 
$\chi_s$: $M\to[0,1]$ be a smooth function 
such that 
$\chi_s\equiv 0$ on $B(p,s)$, 
$\chi_s\equiv 1$ on 
$M\setminus B(p,2s)$ 
and $|\diff\chi_s|\leq\frac{2}{s}$. 
We write $A_s:=B(p,2s)\setminus B(p,s)$ and 
we obtain by (\ref{cut_off_formula}) 
\begin{align*}
\int_{M} u\chi_sP_f(u\chi_s)\,dv^g&=
\int_{M} (u^2|\diff\chi_s|^2+\chi_s^2uP_fu)\,dv^g\\
&\leq \frac{4}{s^2}\int_{A_s}u^2\,dv^g
+\int_{M}\chi_s^2uP_fu\,dv^g.
\end{align*}
Since there exists $C>0$ such that for all $s$ 
we have $\vol(A_s)\leq Cs^n$, the first term on 
the right hand side tends to $0$ as $s\to 0$. 
We conclude that 
\begin{displaymath}
\lim_{s\to 0} J_f(u\chi_s)\leq J_f(u).
\end{displaymath}
Thus we can choose $s$ so close to $0$ 
that we have 
$J_f(u\chi_s)\leq\mu+\varepsilon$. 
Since we have $\chi_s(p)=0$ the left 
hand side is equal to $I_f(u\chi_s)$. 
It follows that $\nu\leq\mu+\varepsilon$. 
\end{proof}

Finally we ask whether the infimum $\nu$ is 
attained. 
We immediately obtain the following answer.

\begin{lemma}
Let $u\in C^{\infty}(M)$ be the unique smooth 
function with $J_f(u)=\mu$ given by Lemma 
\ref{existu}. 
\begin{enumerate}[1.]
\item If $u(p)=0$, then there is exactly one 
$w\in C^{\infty}(M)$ with $w(p)=0$ and 
$I_f(w)=\nu$, namely $w=u$. 
\item If $u(p)\neq0$, then there is no 
$w\in C^{\infty}(M)$ with $w(p)=0$ and 
$I_f(w)=\nu$. 
\end{enumerate}
\end{lemma}

\begin{proof}
If $w\in C^{\infty}(M)$ satisfies 
$w(p)=0$ and $I_f(w)=\nu$ then by 
Lemma \ref{lemma_Ig2} and Lemma \ref{lemma_mu=nu} 
we have $J_f(w)=\mu$ and thus $w=u$. 
Both 1.\,and 2.\,follow from this observation. 
\end{proof}

\section{Another proof of the case $P_f=L_g$}
\label{proof'}

We give an alternative proof of Theorem 
\ref{maintheorem} in the special case 
$f=\frac{n-2}{4(n-1)}\scal_g$. 
Let $(M,g)$ be a closed Riemannian manifold such that 
$g$ is flat on an open neighborhood $U$ of a fixed 
point $p\in M$ and assume that $Y(M,g)>0$. 
Then the mass $m(M,g)$ of $L_g$ at the point $p$ 
is well defined. 
Let $\delta>0$ such that $B(p,\delta)\subset U$ 
and let $\eta$ be a smooth function on $M$ 
such that $\eta\equiv\frac{1}{(n-2)\omega_{n-1}}$ 
on $B(p,\delta)$ and $\supp(\eta)\subset U$, 
where $\omega_{n-1}$ denotes the volume of $S^{n-1}$ 
with the standard metric. 
For every $u\in C^{\infty}(M)$ with $u(p)=0$ 
we define 
\begin{displaymath}
I_g(u):=\int_{M\setminus\{p\}}
(\eta r^{2-n}+u)L_g(\eta r^{2-n}+u)\,dv^g
\end{displaymath}
and
\begin{displaymath}
\nu:=\inf\{I_g(u)|\,u\in C^{\infty}(M),\,
u(p)=0\}. 
\end{displaymath}
We denote by $G$ the Green function for the 
conformal Laplacian $L_g$ at the point $p$. 
It is strictly positive on $M\setminus\{p\}$ 
by Proposition \ref{green_function}. 
Thus $\widetilde{g}:=G^{4/(n-2)}g$ is a 
Riemannian metric on $M\setminus\{p\}$. 
Furthermore for every $u\in C^{\infty}(M)$ 
with $u(p)=0$ the function 
\begin{displaymath}
\Phi_u:=(\eta r^{2-n}+u)G^{-1}
\end{displaymath}
has a smooth extension to all of $M$ and 
in an isometric chart on $U$ it has the expansion 
\begin{equation}
\label{w_tilde_expansion}
\Phi_u(x)=1-Ar^{n-2}+o(r^{n-2})\quad
\textrm{as }x\to p
\end{equation}
with $A:=(n-2)\omega_{n-1}m(M,g)$. 
We prove the following theorem. 

\begin{theorem}
For every $u\in C^{\infty}(M)$ with $u(p)=0$ 
we have 
\begin{equation}
\label{I_g_special_case}
I_g(u)
=\int_{M\setminus\{p\}}
|\diff\Phi_u|_{\widetilde{g}}^2\,
dv^{\widetilde{g}}
-m(M,g). 
\end{equation}
\end{theorem}

\begin{proof}
Let $u\in C^{\infty}(M)$ with $u(p)=0$. 
We write $w:=\eta r^{2-n}+u$. 
By the conformal transformation law 
(\ref{L_g_conform}) for $L_g$ we obtain 
\begin{displaymath}
L_{\widetilde{g}}\Phi_u=
G^{-\frac{n+2}{n-2}}L_gw. 
\end{displaymath}
Since we have $dv^{\widetilde{g}}=G^{2n/(n-2)}dv^g$ 
and since by the conformal transformation law 
(\ref{scalgg'}) for the scalar curvature we have 
$\scal_{\widetilde{g}}=0$ 
it follows that 
\begin{displaymath}
I_g(u)
=\int_{M\setminus\{p\}} w L_g w\, dv^g
=\int_{M\setminus\{p\}}
\Phi_u\Delta_{\widetilde{g}}\Phi_u\,
dv^{\widetilde{g}}.
\end{displaymath}
Integrating by parts, 
 we have for every 
$\varepsilon>0$ 
\begin{displaymath}
\int_{M\setminus B(p,\epsilon)} 
\Phi_u\Delta_{\widetilde{g}}\Phi_u\,
dv^{\widetilde{g}}
=\int_{M\setminus B(p,\epsilon)} 
|\diff\Phi_u|_{\widetilde{g}}^2\,
dv^{\widetilde{g}}
-\int_{\partial B(p,\epsilon)}
\Phi_u\partial_{\widetilde{\nu}}\Phi_u\,
ds^{\widetilde{g}},
\end{displaymath}
where $ds^{\widetilde{g}}$ is the induced volume 
form on $\partial B(p,\epsilon)$ and where 
\begin{displaymath}
\widetilde{\nu}=-G^{-\frac{2}{n-2}}\partial_r
\end{displaymath}
is the outer unit normal vector field on the boundary 
of $M\setminus B(p,\epsilon)$. 
Using (\ref{w_tilde_expansion}) 
we compute the following expansions as $x\to p$
\begin{align*}
\widetilde{\nu}&=
-(\eta^{-\frac{2}{n-2}}r^2+o(r^2))\partial_r\\
\Phi_u\partial_{\widetilde{\nu}}\Phi_u
&=
A(n-2)\eta^{-\frac{2}{n-2}}r^{n-1}+o(r^{n-1})\\
ds^{\widetilde{g}}
&=G^{\frac{2(n-1)}{n-2}}ds^g
=(\eta^{\frac{2(n-1)}{n-2}}r^{-2(n-1)}
+o(r^{-2(n-1)}))ds^g.
\end{align*}
Thus we obtain
\begin{displaymath}
\lim_{\varepsilon\to 0}
\int_{\partial B(p,\epsilon)}
\Phi_u\partial_{\widetilde{\nu}}\Phi_u\,
ds^{\widetilde{g}}
=m(M,g)
\end{displaymath}
and the assertion follows. 
\end{proof}

We now obtain the following special case 
of Theorem \ref{maintheorem}. 

\begin{corollary}
We have $\nu=-m(M,g)$. 
The infimum is attained if and only if $m(M,g)=0$. 
\end{corollary}

\begin{proof}
It follows from (\ref{I_g_special_case}) 
that $\nu=-m(M,g)$ since one can choose 
$u\in C^{\infty}(M)$ with $u(p)=0$ in such a way 
that the first term on the right hand side 
becomes as small as one wants. 
If the infimum is attained at $u\in C^{\infty}(M)$, 
then $\Phi_u$ is constant and by 
(\ref{w_tilde_expansion}) we conclude $m(M,g)=0$. 
On the other hand if $m(M,g)=0$, then 
$G-\eta r^{2-n}$ has a smooth extension $u$ to all 
of $M$ satisfying $u(p)=0$. 
Then with the notation from above we have $w=G$ 
and $\Phi_u=1$ and therefore $I_g(u)=-m(M,g)$.
\end{proof}

\section{Several applications}
\label{applications}

\subsection{Application 1: Positive mass theorem on spin manifolds} 
\label{section_pmt_spin}

Let $(M,g)$ be a closed Riemannian manifold with 
positive Yamabe constant $Y(M,g)$ 
which means that the operator 
$ L_g:= \Delta_g + \frac{n-2}{4(n-1)} \scal_g$ is 
positive (see Paragraph \ref{yamabeop}). 
We assume that $g$ is flat on an open neighborhood 
$U$ of $p \in M$. 
Furthermore we assume in this section that $M$ 
is a spin manifold with a fixed orientation and 
a fixed spin structure. 
We denote by $G$ the Green function of $L_g$ and by 
$m(M,g)$ the associated mass. 
In this section we prove the following positive 
mass theorem for spin manifolds. 

\begin{theorem}
\label{pmt_spin}
Let $(M,g)$ be a closed Riemannian spin manifold 
with positive Yamabe constant $Y(M,g)$ 
such that $g$ is flat on an open 
neighborhood of a point $p\in M$. 
Then we have $m(M,g)\geq0$. 
Furthermore we have $m(M,g)=0$ if and only if 
$(M,g)$ is conformally equivalent to 
$(S^n,\sigma^n)$. 
\end{theorem} 

\noindent This theorem solves the positive mass conjecture  in the particular case of spin manifolds. This was already known  by  the work of Witten \cite{witten:81}. 
Let us come back on the name "mass" used for $m(M,g)$ and more generally for the numbers  $m_f$ associated to the operators $P_f$. 
Set $g':= G^{\frac{4}{n-2}} g$. This new metric is defined on $M \setminus \{ p\}$. As observed by Schoen, the manifold $(M \setminus \{ p \}, g')$ is asymptotically flat. We will not explain in detail what this means, but asymptotically flat manifolds are the standard models for isolated system in general relativity. To each asymptotically flat manifold with positive $L^1$ scalar curvature one can associate a number called the ADM-mass of the manifold which is interpreted as the energy of the isolated system. For this reason, this number should be positive but this is far to be obvious from its mathematical definition. It was proven to be true 
e.\,g.\,on spin manifolds by Witten \cite{witten:81} and in dimension $n \in \{3,..., 7\}$ by Schoen and Yau \cite{schoen.yau:79} but the problem in its full 
generality is still open. 
In the particular case that the asymptotically flat manifold was obtained by blowing-up a closed manifold as above with the Green function of $L_g$ (this procedure is sometimes called {\em stereographic projection} since, starting with a closed manifold  $(M,g)$ conformally equivalent  to the standard sphere, then $(M \setminus \{ p \},g') = (\mR^n, \xi)$), Schoen proved that the number $m(M,g)$ is a positive multiple of the ADM mass of $(M \setminus \{ p \},g')$. This is the reason why the number $m(M,g)$ is called the mass.  In this special context, which is 
actually not restrictive, the positivity of the ADM mass, i.e. on $M$, is also open. 
 Schoen and Yau gave a proof when the manifold is locally conformally flat in \cite{schoen.yau:88}.
 Later, inspired by Witten's proof, Ammann and Humbert gave a very simple proof for spin manifolds which are conformally flat or of dimension $3$, $4$ or $5$ (see \cite{ammann.humbert:05}). 
This last method was adapted to other situations: Jammes \cite{jammes:11} obtained another proof of Schoen-Yau's theorem \cite{schoen.yau:88} for conformally flat manifolds of even dimension and Humbert, Raulot in \cite{humbert.raulot:09} could prove a positive mass theorem for the Paneitz operator.  
The proof we give here is quite similar to the one of Ammann and Humbert and not simpler but it allows to understand the crucial role played by the Green function of the Dirac operator in their proof. Namely, we prove that the norm of this Green function can be used as a test function in the variational characterization of the mass given by Theorem \ref{maintheorem}.  

Before we give the proof we recall some facts from 
spin geometry which we will need. 
Since $M$ is spin, 
for every Riemannian metric $g$ on $M$ we 
can define 
the spinor bundle $\Sigma^gM$ over $M$ which is 
a complex vector bundle of rank $2^{[n/2]}$ 
with a bundle metric $(.,.)$ and a connection 
$\nabla$. 
Smooth sections of $\Sigma^gM$ are called spinors. 
We denote by 
\begin{displaymath}
D_g:\quad \Gamma(\Si^gM)\to\Gamma(\Si^gM)
\end{displaymath}
the Dirac operator acting on spinors. 
For an introduction to the concepts of spin geometry 
the reader may consult the books 
\cite{lawson.michelsohn:89} or \cite{friedrich:00}. 
We will mainly use two important results. 
First, by the Schr\"odinger-Lichnerowicz formula 
we have for all $\psi\in\Gamma(\Si^gM)$ 
\begin{equation}
\label{schroed_lichn}
(D_g)^2\psi=\nabla^*\nabla\psi
+\frac{1}{4}\scal_g\psi,
\end{equation}
where $\nabla^*\nabla$ denotes the connection 
Laplacian on $\Sigma^gM$. 
Second, if $g'=w^{\frac{4}{n-1}}g$ is a metric 
conformal to $g$, where $w$ is a smooth 
positive function on $M$, then by 
\cite{hitchin:74}, \cite{hijazi:86} 
there exists an isomorphism of vector bundles 
\begin{displaymath}
\beta_{g,g'}:\quad\Sigma^gM
\to\Sigma^{g'}M
\end{displaymath}
which is a fiberwise isometry such that 
for all $\psi\in\Gamma(\Sigma^gM)$ we have 
\begin{equation}
\label{dirac_conform}
D_{g'}(w^{-1}\beta_{g,g'}\psi)=
w^{-\frac{n+1}{n-1}}\beta_{g,g'}D_g\psi.
\end{equation}
Furthermore one can show that 
for every element $\psi_0$ of the fiber 
$\Sigma^g_pM$ over $p$ there exists a unique Green 
function of $D_g$, i.\,e.\,a spinor $\psi$ 
on $M\setminus\{p\}$ such that for every 
$\varphi\in\Gamma(\Sigma^gM)$ we have 
\begin{displaymath}
\int_{M \setminus \{p\}} (\psi , D_g \phi) dv^g= 
(\psi_0, \phi(p)).
\end{displaymath}
Using our assumptions 
one can also write down the expansion of $\psi$ 
around $p$ 
similarly as for the Green function of 
$\Delta_g+f$ in Proposition \ref{green_function}. 
Namely we use that $g$ is flat on an open 
neighborhood $U$ 
of $p$ and we choose $\de>0$ such that 
$B(p,\de)\subset U$. 
We may assume that there exists an  
isometric chart 
$B(p,\de)\to B(0,\de)\subset\mR^n$
and that $\Sigma^gM$ 
is trivial on $B(p,\de)$. 
Since $L_g$ is positive, it is well known 
that $D_g$ is invertible. 
Using these facts Ammann and Humbert described the 
expansion of $\psi$ as follows 
(see \cite{ammann.humbert:05}).

\begin{lemma}  
Let $\psi_0\in \Si^g_p M$. 
Then there is a unique spinor 
$\psi$ on $M\setminus \{p\}$ 
such that $D_g\psi=0$ and such that for all 
$x\in B(p,\de)\cong B(0,\de)
\subset\mR^n$ we 
have in the above chart and trivialization 
\begin{equation}
\label{expansion_green_dirac} 
\psi|_{B(p,\de)}(x)= 
-\frac{1}{\omega_{n-1}}
\frac{x}{r^{n}} \cdot \psi_0 
+ \theta(x)
\end{equation}
where $\theta$ is a smooth spinor on $B(p,\de)$.
\end{lemma}           

\noindent 
From now on we assume that 
\begin{equation}
\label{psi0norm}
|\psi_0|= \left( (n-2) \omega_{n-1} \right)^{-\frac{n-1}{n-2}}.
\end{equation} 
Then on $B(p,\de)$ we have 
by (\ref{expansion_green_dirac}) 
\begin{equation}
\label{exp_psi} 
| \psi (x) |^{\frac{n-2}{n-1}}
= \frac{1}{(n-2) \omega_{n-1} r^{n-2}} 
+ o(1)\quad\textrm{as }r\to0.
\end{equation}
\noindent Now, let $\eta$ and $F_{\eta}$ 
be defined as in Section \ref{variational}. 
By Theorem \ref{maintheorem} we have 
\begin{align}
\nonumber 
-m(M,g) &= \inf 
\{I_g(u)|\,u \in C^{\infty}(M),\,u(p)=0\}\\
\label{mass_I_g}
&= \inf 
\{J_g(u)|\,u\in C^{\infty}(M)\}
\end{align}  
where
\begin{align*}
I_g(u) &= 
\int_{M \setminus \{ p \} } 
(\eta r^{2-n} +u) L_g(   \eta r^{2-n} +u)\,dv^g ,\\
J_g(u) &= 
\int_{M\setminus\{p\}}
\eta r^{2-n}F_{\eta}\,dv^g
+2\int_{M}
uF_{\eta}\,dv^g
+\int_{M}uL_g u\,dv^g.
\end{align*} 
The function 
\begin{equation}
\label{u_test_function}
u:\quad M\to\mR,\quad u(x):=
\left\{
\begin{array}{ll}
|\psi(x)|^{\frac{n-2}{n-1} } - \eta(x) r(x)^{2-n}, &
\textrm{if }x\neq p\\
0, &\textrm{if }x=p
\end{array} 
\right.
\end{equation}
is smooth on the complement of the zero set of $\psi$. 

The idea for our proof of Theorem 
\ref{pmt_spin} is to use the characterization 
(\ref{mass_I_g}) of $m(M,g)$ and to use $u$ 
as a test function for our functional $I_g$. 
If $\psi$ has non-empty zero set, 
then $u$ is not smooth and we will approximate $u$ 
by a sequence of smooth functions. 
Since the zero set of $\psi$ has Hausdorff dimension 
at most $n-2$ (see \cite{baer:97}), the proof 
will also work in this case. 
This is the content of the following proposition. 

\begin{proposition}
\label{proposition_pmt_spin}
There exists a sequence $(u_k)_{k\in\mN}$ of smooth 
functions on $M$ such that $u_k(p)=0$ for all $k$ 
and $\lim_{k\to\infty}J_g(u_k)=
\lim_{k\to\infty}I_g(u_k)\leq 0$.
\end{proposition}

\begin{proof}
We first write down the proof in the case that $\psi$ 
is nowhere zero and consider the case of non-empty 
zero set afterwards. 
If $\psi$ is nowhere zero 
then $g':= |\psi|^{\frac{4}{n-1}} g$ 
is a Riemannian metric on $M\setminus\{p\}$. 
As explained above 
there exists an isomorphism of vector bundles 
\begin{displaymath}
\beta_{g,g'}:\quad\Sigma^g(M\setminus\{p\})
\to\Sigma^{g'}(M\setminus\{p\})
\end{displaymath}
which is a fiberwise isometry. 
Furthermore 
with $\psi':= |\psi|^{-1}\beta_{g,g'}\psi$ 
we have $D_{g'} \psi' = 0$ by (\ref{dirac_conform}). 
Let $\ep>0$ be small. In what follows, the set 
$B(p,\ep)$ is the ball of center $p$ and radius $\ep$ 
for the metric $g$. 
By (\ref{schroed_lichn}) we have 
\begin{equation} 
\label{e1} 
0   = \int_{M \setminus B(p,\ep)} 
(D^2_{g'} \psi',\psi') \, dv^{g'} 
= \int_{M \setminus B(p,\ep)} 
\Big(( \nabla^* \nabla \psi', \psi') 
+ \frac{1}{4} \scal_{g'} |\psi'|^2 \Big)dv^{g'} . 
\end{equation} 
Note that $|\psi'| \equiv 1$. 
Hence, integrating by parts:
\begin{align} 
\nonumber 
\int_{M \setminus B(p,\ep)} 
( \nabla^* \nabla \psi', \psi') dv^{g'} & = 
\int_{M \setminus B(p,\ep)} | \nabla \psi' |^2 dv^{g'} 
+\int_{\partial  B(p,\ep)} 
(\nabla_{\nu}  \psi', \psi') ds^{g'} \\
\nonumber 
&  =   
\int_{M \setminus B(p,\ep)} | \nabla \psi' |^2 dv^{g'} 
+ \frac{1}{2} \int_{\partial  B(p,\ep)} 
\partial_{\nu}  | \psi'|^2 ds^{g'} \\
\label{e2} 
& =   \int_{M \setminus B(p,\ep)} | \nabla \psi' |^2 dv^{g'}. 
\end{align}
where $\nu$ is the outer unit normal vector field on $B(p,\ep)$ and $ds^{g'}$ is the volume element induced by $g'$ on $\partial B(p,\ep)$. 
By Equation (\ref{scalgg'}), we also have 
\begin{displaymath}
\scal_{g'} = \frac{4(n-1)}{n-2}
|\psi|^{-\frac{n+2}{n-1}}
L_g ( |\psi|^{\frac{n-2}{n-1}}).
\end{displaymath}
Since $dv^{g'} = |\psi|^{ \frac{2n}{n-1}}dv^g$, we obtain that 
\begin{displaymath}
\int_{M \setminus B(p,\ep)}\scal_{g'} |\psi'|^2 dv^{g'} 
= \frac{4(n-1)}{n-2}
\int_{ M \setminus B(p,\ep)}  
|\psi|^{\frac{n-2}{n-1}}
L_g ( |\psi|^{\frac{n-2}{n-1}}) dv^g.
\end{displaymath}
Taking the limit as $\ep$ tends to $0$, we obtain
\begin{align*}
\lim_{\ep \to 0} 
\int_{M \setminus B(p,\ep)}
\scal_{g'} |\psi'|^2 dv^{g'} &= 
\frac{4(n-1)}{n-2} 
\int_{ M \setminus \{ p \} }  
|\psi|^{\frac{n-2}{n-1}}
L_g ( |\psi|^{\frac{n-2}{n-1}}) dv^g \\ &= 
\frac{4(n-1)}{n-2}I_g (u)
\end{align*} 
where $u$ is defined 
in (\ref{u_test_function}). 
Together with (\ref{e1}) and (\ref{e2}) we obtain 
\begin{equation}
\label{I_leq_0}
0 =  \int_{M \setminus \{ p \} } | \nabla \psi' |^2 dv^{g'} + \frac{n-1}{n-2} I_g(u)
\end{equation} 
which implies $I_g(u) \leq 0$. 
Furthermore by (\ref{exp_psi}) we have $u(p)=0$ 
and by Lemma \ref{lemma_Ig2} it follows that 
$J_g(u)=I_g(u)$. 
This finishes the proof if $\psi$ is 
nowhere zero.\\

\noindent If $\psi$ has non-empty zero set $N$, 
then for every $s>0$ we define 
\begin{displaymath}
B_s(N):=\{x\in M|\,\dist_g(x,N)<s\}
\end{displaymath}
and for every $k\in\mN$ we define 
\begin{displaymath}
M_k:=\Big\{x\in M\big|\,
\dist_g(x,N)>\frac{2}{k}\Big\}.
\end{displaymath} 
Then the calculation (\ref{e1}) holds 
with $M_k$ instead of $M$. 
If we do the calculation (\ref{e2}) with $M_k$ 
instead of $M$ then we obtain an extra 
boundary term 
\begin{displaymath}
\int_{\partial  M_k} (\nabla_{\nu}  \psi', \psi') ds^{g'}
\end{displaymath} 
which vanishes since $|\psi'|\equiv 1 $. 
Thus we conclude 
\begin{equation}
\label{e3}
0 =  \int_{M_k\setminus\{p\}} | \nabla \psi' |^2 dv^{g'} + \frac{n-1}{n-2} 
\int_{M_k\setminus\{p\}}|\psi|^{\frac{n-2}{n-1}} L_g(|\psi|^{\frac{n-2}{n-1}})\,dv^g.
\end{equation}
For every $k\in\mN$ we choose a smooth 
function $\chi_k$: $M\to[0,1]$ such that 
$\chi_k(x)= 0$ if $\dist_g(x,N)\leq\frac{1}{k}$, 
$\chi_k(x)= 1$ if $\dist_g(x,N)\geq\frac{2}{k}$ 
and $|\diff\chi_k|_g \leq 2k$ 
and we define $u_k:=\chi_k u$ 
and $A_k:=\{x\in M|\,\frac{1}{k}<\dist_g(x,N)<\frac{2}{k}\}$. 
Then we have 
\begin{align}
\nonumber 
&\int_{M_k\setminus\{p\}}|\psi|^{\frac{n-2}{n-1}} L_g(|\psi|^{\frac{n-2}{n-1}})\,dv^g\\
\label{e4}
=& I_g(u_k)-
\int_{A_k}
\chi_k |\psi|^{\frac{n-2}{n-1}} 
L_g(\chi_k |\psi|^{\frac{n-2}{n-1}})\,dv^g.
\end{align}
Next we define $\nu$ as the outer unit normal vector 
field on $\partial A_k$ and we obtain 
\begin{align}
\nonumber 
&\int_{A_k}
\chi_k |\psi|^{\frac{n-2}{n-1}} 
\Delta_g(\chi_k |\psi|^{\frac{n-2}{n-1}})\,dv^g\\
\nonumber 
&= \int_{A_k}
|\diff (\chi_k |\psi|^{\frac{n-2}{n-1}})|^2\,dv^g
-\int_{\partial A_k}
\chi_k |\psi|^{\frac{n-2}{n-1}}
\partial_{\nu}(\chi_k |\psi|^{\frac{n-2}{n-1}})\,ds^g\\
\label{e5}
&= \int_{A_k}
|\diff (\chi_k |\psi|^{\frac{n-2}{n-1}})|^2\,dv^g
-\int_{\partial B_{2/k}(N)}
|\psi|^{\frac{n-2}{n-1}}
\partial_{\nu}|\psi|^{\frac{n-2}{n-1}}\,ds^g.
\end{align}
In order to estimate the derivatives of 
$|\psi|^{(n-2)/(n-1)}$ near $N$ we note that for all 
$Y\in T(M\setminus\{p\})$ we have
\begin{displaymath}
\partial_Y|\psi|^{\frac{n-2}{n-1}}
=\frac{n-2}{n-1}|\psi|^{-\frac{n}{n-1}}
\mathrm{Re}(\nabla_Y\psi,\psi).
\end{displaymath}
Thus there exists $C_1>0$ such that for all 
$k\in\mN$ large enough, for all 
$x\in\overline{B_{2/k}(N)}$ and for all $Y\in T_xM$ 
with $|Y|=1$ we have the estimate 
\begin{displaymath}
\Big|\partial_Y|\psi|^{\frac{n-2}{n-1}}(x)\Big|
\leq C_1|\psi(x)|^{-\frac{1}{n-1}}.
\end{displaymath}
Since $|\psi|^2$ is a $C^1$-function there exists 
$C_2>0$ such that for all $k\in\mN$ large enough 
and for all $x\in\overline{B_{2/k}(N)}$ 
we have $|\psi(x)|^2\leq C_2\dist_g(x,N)$. 
Thus there exists $C_3>0$ such that for all 
$k\in\mN$ large enough, for all 
$x\in\overline{B_{2/k}(N)}$ and for all $Y\in T_xM$ 
with $|Y|=1$ we have 
\begin{equation}
\label{e6} 
\Big|\partial_Y|\psi|^{\frac{n-2}{n-1}}(x)\Big|
\leq C_3k^{\frac{1}{2(n-1)}}.
\end{equation}
Furthermore since $N$ has Hausdorff dimension at 
most $n-2$ there exists $C_4>0$ such that for all 
$k\in\mN$ large enough we have 
\begin{equation}
\label{e7} 
\vol(A_k)\leq \frac{C_4}{k^2},\quad 
\vol(\partial B_{2/k}(N))\leq\frac{C_4}{k}.
\end{equation}
Using (\ref{e6}), (\ref{e7}) 
and using that $|\diff\chi_k|_g\leq 2k$ 
we obtain from (\ref{e5}) that 
\begin{displaymath}
\int_{A_k}
\chi_k |\psi|^{\frac{n-2}{n-1}} 
L_g(\chi_k |\psi|^{\frac{n-2}{n-1}})\,dv^g\to 0
\end{displaymath}
as $k\to\infty$. 
Therefore we obtain from (\ref{e3}), (\ref{e4}) 
that $\liminf_{k\to\infty} I_g(u_k)\leq 0$. 
Furthermore by (\ref{exp_psi}) we have $u_k(p)=0$ 
for all $k$ and by Lemma \ref{lemma_Ig2} it follows 
that $J_g(u_k)=I_g(u_k)$ for all $k$. 
This finishes the proof in the general case. 
\end{proof}


\begin{proof}[Proof of Theorem \ref{pmt_spin}]
The first statement follows immediately 
from Proposition \ref{proposition_pmt_spin} 
and from Lemma \ref{lemma_mu=-m}. 

Next let $m(M,g)=0$ and let $(u_k)_{k\in\mN}$ be the 
sequence in $C^{\infty}(M)$ constructed in 
the proof of Proposition \ref{proposition_pmt_spin}. 
We have $\liminf_{k\to\infty}J_g(u_k)=0$ 
and therefore there exists a subsequence of 
$(u_k)_{k\in\mN}$ which is a minimizing sequence 
for the functional $J_g$. 
From the proof of Lemma \ref{existu} it follows 
that after passing again to a subsequence the 
sequence $(u_k)_{k\in\mN}$ converges pointwise 
almost everywhere to the minimizer $G-\eta r^{2-n}$ 
of the functional $J_g$. 
Therefore we have 
\begin{displaymath}
|\psi|^{\frac{n-2}{n-1}}=G
\end{displaymath}
almost everywhere on $M\setminus\{p\}$ and since 
both functions are continuous the equality 
holds everywhere on $M\setminus\{p\}$. 
By Proposition \ref{green_function} 
the function $G$ is 
strictly positive on $M\setminus\{p\}$. 
In particular $\psi$ is nowhere zero and 
$|\psi|$ and the metric $g'$ constructed 
in the proof of Proposition \ref{proposition_pmt_spin} 
are independent of the choice of 
$\psi_0\in\Sigma_pM$ satisfying (\ref{psi0norm}). 
For every such spinor $\psi_0\in\Sigma_pM$ 
the spinor $\psi'$ constructed from $\psi_0$ 
as in the proof of 
Proposition \ref{proposition_pmt_spin} 
is a parallel spinor for the metric $g'$ 
by (\ref{I_leq_0}). 
Since the choice of $\psi_0$ 
is arbitrary we obtain a trivialization of 
the spinor bundle $\Sigma^{g'}(M\setminus\{p\})$ 
by parallel spinors. 
As in the proof of Theorem 2.2 
in \cite{ammann.humbert:05} it follows that 
$(M,g)$ is conformally equivalent to 
$(S^n,\sigma^n)$.
\end{proof}

\subsection{Application 2: A mass-to-infinity theorem}

Let $M$ be a closed Riemannian manifold of dimension $n \geq 3$. We consider a sequence $g_k$ of metrics which converges in $C^2(M)$ to a metric $g_\infty$. 
We assume that all the metrics $g_k$, $g_\infty$ are flat on a fixed neighborhood $U$ of $p \in M$. 
Let also $(f_k)_{k\in\mN}$ be a sequence in 
$C^{\infty}(M)$ such that 
for every $k$ we have $f_k\equiv 0$ on $U$ 
and such that for every $k$ all eigenvalues of 
the operator $P_k:=\Delta_{g_k}+f_k$ are positive. 
Furthermore we assume that there exists 
$f_{\infty}\in C^{\infty}(M)$ such that 
$f_k\to f_{\infty}$ in $C^{\infty}(M)$ 
and we write 
$P_{\infty}:=\Delta_{g_\infty}+f_{\infty}$. 
Note that we just prove the result for $C^\infty$ for 
simplicity but these assumptions could easily be 
weakened. 
For every $k\in\mN$ the Green function of the 
operator $P_k$ has an expansion as in 
Proposition \ref{green_function} and we will 
denote the mass of $P_k$ by $m_k$.

\begin{theorem} 
\label{massinfty}
Assume that the first eigenvalue $\la_\infty$ of 
$P_\infty$ is $0$.  
Then we have $m_k\to\infty$ as $k\to\infty$. 
\end{theorem}
This is a generalization of a result by Beig and 
O'Murchadha \cite{beig.omurchada:97} who proved it 
with $f_k=\frac{n-2}{4(n-1)} \scal_{g_k}$,  i.e. $P_k$ 
is the Yamabe operator of the metric $g_k$. 
The limiting metric $g_\infty$ was assumed to have a 
vanishing Yamabe constant (i.e. the first eigenvalue 
of $P_\infty= L_{g_\infty}$ is equal to $0$). 
With the use of Theorem \ref{maintheorem}, the proof 
is much simpler than the proof by Beig and 
O'Murchadha.  
\begin{proof}
Let $k\in\mN$. 
We choose $\delta_k>0$ such that 
the ball $B(p,\delta_k)$ centered at $p$ and 
of radius $\delta_k$ with respect to the metric 
$g_k$ is contained in $U$. 
Then we define a smooth non-negative function 
$\eta_k$ on $M$ such that 
$\eta_k\equiv\frac{1}{(n-2)\omega_{n-1}}$ on 
$B(p,\delta_k)$ and such that 
$\supp(\eta_k)\subset U$. 
For $x\in M$ let $r_k(x)$ denote the distance of $x$ 
to $p$ with respect to the metric $g_k$. 
The function $F_{\eta_k}$: $M\to\mR$ defined by 
\begin{displaymath}
F_{\eta_k}(x)=
\left\{\begin{array}{ll}
\Delta_{g_k}(\eta_k r_k^{2-n})(x),&x\neq p\\
0, & x=p
\end{array}\right.
\end{displaymath}
is smooth on $M$. 
For every $u\in C^{\infty}(M)$ we define 
\begin{displaymath}
J_k(u):=
\int_{M\setminus\{p\}}
\eta_k r_k^{2-n}F_{\eta_k}\,dv^{g_k}
+2\int_{M}
uF_{\eta_k}\,dv^{g_k}
+\int_{M}uP_ku\,dv^{g_k}.
\end{displaymath}
Then by Theorem \ref{maintheorem} 
we have for all $k\in\mN$ 
\begin{displaymath}
-m_k=\inf\{J_k(u)|\,u\in C^{\infty}(M)\}.
\end{displaymath}
Let $u \in C^{\infty}(M)$ be an eigenfunction 
associated to $\la_\infty$. 
It is a classical result that the eigenfunctions 
corresponding to the first eigenvalue of an 
operator of the form $P_f$ are either strictly 
positive or strictly negative. 
Thus we may assume that $u$ is strictly positive. 
As in the proof of Lemma \ref{lemma_Ig2} 
one obtains for every $k$: 
\begin{displaymath}
\int_{M}uF_{\eta_k}\,dv^{g_k}
=\int_{M\setminus\{p\}}u
\Delta_{g_k}(\eta_k r_k^{2-n})\,dv^{g_k}
=\int_{M\setminus\{p\}}
\eta_k r_k^{2-n}\Delta_{g_k} u\,dv^{g_k}
-u(p).
\end{displaymath}
Since $g_k\to g_{\infty}$ in $C^2(M)$ we have 
$\Delta_{g_k}u\to\Delta_{g_{\infty}}u$ in $C^0(M)$. 
Since $P_{\infty}u\equiv 0$ and 
since $f_{\infty}\equiv 0$ 
on $U$ it follows that 
$\Delta_{g_\infty} u  \equiv 0$ on $U$. 
Since $\supp(\eta_k)\subset U$ we conclude that 
\begin{displaymath}
\int_{M}uF_{\eta_k}\,dv^{g_k}\to -u(p)<0
\end{displaymath}
as $k\to\infty$. 
Since $P_ku\to P_{\infty}u=0$ in $C^0(M)$ 
it follows that 
\begin{displaymath}
a_k:=\int_{M}uP_ku\,dv^{g_k}\to0
\end{displaymath}
as $k\to\infty$. 
Now for every $k\in\mN$ we have 
\begin{displaymath}
-m_k\leq J_k(a_k^{-1/3}u)
=\int_{M\setminus\{p\}}
\eta_k r_k^{2-n}F_{\eta_k}\,dv^{g_k}
+2a_k^{-1/3}\int_{M} uF_{\eta_k}\,dv^{g_k}
+a_k^{-2/3}a_k
\end{displaymath}
and the right hand side tends to $-\infty$ 
as $k\to\infty$. 
The assertion follows. 
\end{proof}

\subsection{Application 3: Real analytic families of masses and negative mass}
In this section, we study the family of masses  associated to a family of operators of the type $\Delta_g + f$. As an application, we prove that on any manifold, there exists a function $f$ such that the operator $\Delta_g + f$ is positive but with negative mass. This shows in particular that a proof of a positive mass theorem as studied in Section \ref{section_pmt_spin} must use the conformal properties of the operator $L_g$.  
Let $(M,g)$ be a closed Riemannian manifold such that 
$g$ is flat on an open neighborhood $U$ of a 
point $p\in M$. 
Let $\varphi,f\in C^{\infty}(M)$ 
such that $f\equiv 0$ and $\varphi\equiv 0$ 
on $U$. 
For every $a\in\mR$ we define the operator 
$P_a:=\Delta_g+f+a\varphi$. 
We assume that for $a=0$ 
all eigenvalues of $P_0$ are positive. 
Since the operator $P_0$ is invertible, it follows 
from the Neumann series expansion of the inverse 
that there exists an open interval $I$ containing 
$0$ such that for every $a\in I$ the operator 
$P_a$ is invertible 
(see e.\,g.\,\cite[IV-1.16]{kato:95}).
Since by a theorem of Rellich the eigenvalues 
of $P_a$ are real analytic functions of $a$ 
(see \cite[VII-3.9]{kato:95}), it follows that 
for every $a\in I$ the operator $P_a$ has only 
positive eigenvalues. 
Moreover we can choose $I$ as the maximal interval 
with this property. 
For every $a\in I$ we can define the mass of $P_a$ 
and we denote it by $m(a)$. 
Furthermore, for every $a\in\mR$ and for every 
$u\in C^{\infty}(M)$ we define 
\begin{align*}
I_a(u)&:=
\int_{M\setminus\{p\}}
(\eta r^{2-n}+u)P_a(\eta r^{2-n}+u)\,dv^g,\\
J_a(u)&:=
\int_{M\setminus\{p\}}
\eta r^{2-n}F_{\eta}\,dv^g
+2\int_{M}
uF_{\eta}\,dv^g
+\int_{M}uP_au\,dv^g,
\end{align*}
where $\eta$ and $F_{\eta}$ are as in Section 
\ref{variational}.  
By Theorem \ref{maintheorem}, we have 
\begin{align*}
-m(a) &= 
\inf\{I_a(u)|\,u \in C^{\infty}(M),\, u(p)= 0\}\\
&=\inf\{J_a(u)|\, u \in C^{\infty}(M)\}.
\end{align*} 
The main result of this section is the 
following theorem. 

\begin{theorem} 
\label{massfunction_theorem}
\begin{enumerate}[1.]
\item The function $I\to\mR$, $a\mapsto m(a)$ 
is real analytic.
\item  The function $I\to\mR$, $a\mapsto m(a)$ 
is convex. 
\item Assume that there exists a point $q\in M$ 
such that $\varphi(q)<0$. 
Then there exists $a_{\infty}>0$ such that $m(a)$ 
can be defined for all $a\in[0,a_{\infty})$ 
and we have $m(a)\to\infty$ as $a\to a_{\infty}$.
\item If $\varphi\geq0$, then $m(a)$ can be defined  
for all $a\geq0$, the function $a\mapsto m(a)$  is 
non-increasing and we have 
\begin{align*}
\lim_{a\to\infty}m(a)& = 
-\inf\{J_0(u)|\,u\in C^{\infty}(M),\, 
\supp(u)\subset M\setminus\supp(\varphi)
\} \\
& = 
- \inf \{ I_0(u)|\,u\in C^{\infty}(M),\, 
u(p)=0,\, 
\supp(u)\subset M\setminus\supp(\varphi)
\} \\
& =: m_{f, M\setminus \supp(\varphi) } 
> - \infty.
\end{align*}
\end{enumerate}
\end{theorem}

\noindent
\begin{corollary}
\label{massfunctioncor}
There exists a function $f$ such that $P_f$ is 
positive and such that $m_f <0$.
\end{corollary}

\begin{corollary}
\label{massfunctioncor2} 
Let $p\in S^n$. 
There exists a Riemannian metric $g$ on $S^n$ 
which is conformal to $\sigma^n$ and 
flat on an open neighborhood of $p$ 
such that for the operator $P_a:=\Delta_g+a\scal_g$ 
we have  $m(a)<0$ for all $a>\frac{n-2}{4(n-1)}$. 
\end{corollary}

\subsubsection{Proof of Theorem \ref{massfunction_theorem} Point 1.}
For every $a\in I$ we denote the Green 
function for the operator $P_a$ by $G_a$. 
We have 
\begin{displaymath}
(P_0+a\varphi)G_a=\delta_p,\quad P_0G_0=\delta_p
\end{displaymath}
and therefore 
\begin{equation}
\label{m_a_analytic}
(P_0+a\varphi)(G_a-G_0)=-a\varphi G_0,
\end{equation}
where the right hand side is smooth, since $\varphi$ 
vanishes on an open neighborhood of $p$. 
The family of bounded linear operators 
\begin{displaymath}
I\ni a\mapsto P_0+a\varphi\in B(C^2(M),C^0(M))
\end{displaymath}
is real analytic and for every $a\in I$ 
the operator $P_0+a\varphi$ is invertible. 
It follows that the family of bounded linear 
operators 
\begin{displaymath}
I\ni a\mapsto (P_0+a\varphi)^{-1}\in 
B(C^0(M),C^2(M))
\end{displaymath}
is real analytic as well 
(see \cite[VII-\S 1.1]{kato:95}). 
From (\ref{m_a_analytic}) we obtain that 
the family of smooth functions $a\mapsto G_a-G_0$ 
is real analytic. 
The assertion follows.

\subsubsection{Proof of Theorem \ref{massfunction_theorem} Point 2.} 
Denote by $G_a':= \frac{d}{da} G_a$ and 
$G_a'':= \frac{d^2}{da^2} G_a$. 
Differentiating  twice $P_a G_a = \delta_p$, we get: 
\begin{equation} 
\label{conv1}  
P_a G_a' = - \varphi G_a  \; \hbox{ and } \; 
P_a G_a'' = -2 \varphi G_a'. 
\end{equation} 
Now, observe that $G_a''(p) = m''(a)$. 
As a consequence, since in the sense of distributions 
$P_aG_a = \delta_p$ and using (\ref{conv1}), 
we have 
\begin{align*}
m''(a ) &  = 
\int_{M \setminus \{ p \}} G_a P_a G_a''\, dv^{g}\\
&  = -2 \int_{M \setminus \{ p \} } 
\varphi G_a' G_a dv^g \\
& = 2 \int_{M \setminus \{ p \}} G_a' P_a G_a'\, dv^g 
\geq 0. 
\end{align*}
The last inequality comes from the fact that $G_a'$ is 
smooth on $M$ and that $P_a$ is a positive operator.

\subsubsection{Proof of Theorem \ref{massfunction_theorem} Point 3.}
Denote by $\la_a$ the first eigenvalue of $P_a$. By assumption, $\la_0>0$.
Since $\varphi(q)<0$ there exists an open neighborhood 
$V\subset M$ of $q$ such that $\varphi<0$ on $V$. 
Let $v \not= 0$ be a non-negative function supported 
in $V$. 
Then, for $a$ large enough, $\int_M v P_av dv^{g} <0$ 
and hence $\la_a <0$. 
Define $a_\infty$ as
\begin{displaymath}
a_\infty:= \inf\{ a>0 |\, \la_a = 0\}.
\end{displaymath}
Then, by Theorem \ref{massinfty} we have 
$m(a) \to \infty$ as $a\to a_{\infty}$.

\subsubsection{Proof of Theorem \ref{massfunction_theorem} Point 4.} 

Since $P_0$ is a positive operator and since 
$\varphi\geq 0$ we have for all $a\geq 0$ and 
for all $u\in C^{\infty}(M)$ with $u\not\equiv 0$ 
\begin{displaymath}
\int_{M} uP_au\,dv^g
=\int_{M} uP_0u\,dv^g
+a\int_{M} \varphi u^2\,dv^g
>0.
\end{displaymath}
Thus for all $a\geq 0$ the operator $P_a$ is 
positive and $m(a)$ can be defined. 

For every $a\geq 0$ and for every 
$u\in C^{\infty}(M)$ we have 
\begin{displaymath}
J_a(u)=J_0(u)+a\int_{M} \varphi u^2\,dv^g,
\end{displaymath}
where the integral on the right hand side is 
non-negative. 
Let $a_1,a_2\geq0$ with $a_1\leq a_2$. 
Then for every $u\in C^{\infty}(M)$ we have 
$J_{a_1}(u)\leq J_{a_2}(u)$. 
It follows that $m(a_1)\geq m(a_2)$ 
and thus the function $a\mapsto m(a)$
is non-increasing. 
Next let $u\in C^{\infty}(M)$ such that 
$\supp(u)\subset M\setminus\supp(\varphi)$. 
Then for all $a\geq 0$ we have 
$-m(a)\leq J_a(u)=J_0(u)$. 
Since this holds for every $u\in C^{\infty}(M)$ 
such that 
$\supp(u)\subset M\setminus\supp(\varphi)$, 
we obtain 
\begin{displaymath}
-m(a)\leq\inf\{J_0(u)|\,u\in C^{\infty}(M),\,
\supp(u)\subset M\setminus\supp(\varphi)\}.
\end{displaymath}
Thus the function $a\mapsto m(a)$ is bounded 
from below and the limit $\lim_{a\to\infty}m(a)$ 
exists. \\

\noindent 
In the following we may assume without loss of 
generality that $\varphi\not\equiv 0$. 
We now need to obtain some properties of $G_a$. 
Let us observe that $G_0 -G_a$ is smooth. 
One computes that 
\begin{equation} 
\label{g0ga} 
P_0 (G_0-G_a) = a \varphi G_a.
\end{equation}
Multiplying this equation by the Green function 
of $P_0$ at any point $q\in M\setminus\{p\}$ 
and integrating we obtain $(G_0-G_a)(q)>0$. 
It follows that $0 < G_a < G_0$ on 
$M\setminus\{p\}$. 
Therefore, since 
\begin{displaymath}
1 = \int_{M\setminus\{p\}} 
P_a(1) G_a\,dv^g 
= \int_M f G_a\,dv^g 
+ a \int_M \varphi G_a\,dv^g
\end{displaymath}
we obtain that 
\begin{equation} 
\label{aga}
a \int_M \varphi G_a\,dv^g \leq C
\end{equation}
for some fixed positive constant $C$ 
which is independent of $a$. 
We multiply (\ref{g0ga}) by $G_0 -G_a$ and 
integrate.   
\begin{align*}
a \int_M \varphi G_a G_0\, dv^g  
& \geq a \int_M \varphi G_a G_0\, dv^g 
- a \int_M \varphi G_a^2\,  dv^g \\
&  =  \int_{M} (G_0-G_a) P_0(G_0-G_a)\, dv^g \\
&= \int_{M} |\diff (G_0 -G_a)|^2 dv^g 
+ \int_M f (G_0-G_a)^2 dv^g
\end{align*}
and the right hand side is positive since $P_0$ 
is a positive operator. 
From (\ref{aga}), we deduce that 
$ a \int_M \varphi G_a G_0 dv^g $ is bounded, and 
hence the same holds for 
\begin{displaymath}
\int_{M} |\diff (G_0 -G_a)|^2 dv^g 
+ \int_M f (G_0-G_a)^2 dv^g.
\end{displaymath}
This implies that $G_0 - G_a$ is bounded in the 
Sobolev space $H^{1,2}(M)$. 
Hence, there exists a function 
$v_\infty \in H^{1,2}(M)$ such that after taking a 
subsequence the functions $G_0-G_a$ tend to 
$v_\infty$ weakly in $H^{1,2}(M)$ and strongly in 
$L^2(M)$.  
We now set $u_a:= G_a - \eta r^{2-n}$. 
Then $u_a$ tends to 
$u_\infty:= -v_\infty+ G_0 - \eta r^{2-n}$ weakly 
in $H^{1,2}(M)$ and strongly in $L^2(M)$ and 
pointwise almost everywhere. 
Observe that $u_\infty$ is non-negative 
on $\supp(\varphi)$ since $u_a \equiv G_a$ on 
$\supp(\varphi)$. 
Moreover, by (\ref{aga}) we have 
\begin{displaymath}
\int_M \varphi u_\infty\, dv^g = 
\lim_{a\to\infty} \int_M \varphi G_a\,dv^g = 0
\end{displaymath}
and as a consequence, 
$u_\infty \equiv 0$ on $\supp(\varphi)$. \\

\noindent For all smooth functions $u$ we have 
\begin{displaymath}
J_0(u) = 
\int_M (|\diff u|^2 + f u^2)\,dv^g 
+ \int_{M \setminus \{p \} } 
\eta r^{2-n} F_{\eta}\, dv^g 
+ 2 \int_M u F_{\eta}\, dv^g.
\end{displaymath}
By density of $C^{\infty}(M)$ in $H^{1,2}(M)$ and 
since $u_\infty$ vanishes on $\supp(\varphi)$, 
we thus have 
\begin{align*} 
&\inf\{J_0(u)|\, u\in C^{\infty}(M),\,
\supp(u)\subset M\setminus\supp(\varphi)\}  \\ 
&= \int_M 
(|\diff u_\infty|^2 + f u_\infty^2)\, dv^g 
+ \int_{M \setminus \{p \} } \eta r^{2-n} F_{\eta}\,  dv^g 
+ 2 \int_M u_\infty F_{\eta}\, dv^g.
\end{align*}

\noindent By weak convergence in $H^{1,2}(M)$ and strong convergence in $L^2(M)$
of $u_a$ to $u_\infty$, it follows that the right hand side is bounded above by (see the proof of Lemma \ref{existu} for details)
\begin{align*} 
&\liminf_{a \to \infty}  
\int_M (|\diff u_a|^2 +  f u_a^2)\, dv^g 
+ \int_{M \setminus \{p \} } \eta r^{2-n} F_{\eta}\, dv^g 
+ 2 \int_M u_a F_{\eta}\, dv^g  \\
&= \liminf_ {a \to \infty} J_0(u_a).
\end{align*}

\noindent This implies that 
\begin{equation} 
\label{uinfty}
\inf\{J_0(u)|\, u\in C^{\infty}(M),\,
\supp(u)\subset M\setminus\supp(\varphi)\}  
\leq \liminf_{a \to \infty} J_0(u_a).
\end{equation}

\noindent From Theorem \ref{maintheorem}, 
\begin{displaymath} 
-m(a) = J_a(u_a)= J_0(u_a) 
+ a \int_M \varphi G_a^2\,dv^g \geq J_0(u_a) 
\end{displaymath} 
which gives, together with (\ref{uinfty}) that 
\begin{displaymath}
\inf\{J_0(u)|\, u\in C^{\infty}(M),\,
\supp(u)\subset M\setminus\supp(\varphi)\}  
\leq  - \lim_{a\to\infty} m(a).
\end{displaymath}
This proves Point 4 of Theorem \ref{massfunction_theorem}.

\subsubsection{Proof of Corollary \ref{massfunctioncor}}
Let us for a moment consider the sphere $S^n$. 
Let $h$ be a metric on $S^n$ which is conformal to 
the standard metric and which is flat on a ball 
$B(q,\de)$ of radius $\delta$ for some 
$q\in S^n$ where $\delta>0$ is chosen such that 
$(M,g)$ is flat on $B(p,\delta)$. 
Let $\overline{\varphi}$ be a smooth function on 
$S^n$ which is positive on 
$S^n\setminus \overline{B(q,\de)}$ 
and which vanishes on $B(q,\delta)$. 
For every $a\geq0$ let $G_a$ be the Green function 
of the operator $L_h+a\overline{\varphi}$ 
and let $\overline{m}(a)$ be its mass. 
We have 
\begin{displaymath}
L_h (G_0 -G_a) = a \overline{\varphi} G_a.
\end{displaymath} 
As in the lines after Equation (\ref{g0ga}) 
it follows that 
for all $a>0$ we have 
$G_0 - G_a > 0$ and hence, 
$\overline{m}(a) = \overline{m}(a) - \overline{m}(0) = (G_a - G_0)(q) <0$. 
By Point 4 of Theorem \ref{massfunction_theorem} 
the function $a \mapsto \overline{m}(a)$ is 
non-increasing. 
Hence, $\lim_{a\to \infty } \overline{m}(a) < 0$. 
Applying Point 4 of Theorem 
\ref{massfunction_theorem}, 
we obtain that 
\begin{equation} 
\label{limonsphere}
\inf\{J(u)|\, u\in C^{\infty}(S^n),\, 
\supp(u)\subset B(q,\delta) \}  > 0
\end{equation} 
where $J$ is defined as above on the sphere by 
\begin{displaymath}
J(u):=
\int_{S^n \setminus\{q\}}
\eta r^{2-n}F_{\eta}\,dv^g
+2\int_{S^n}
uF_{\eta}\,dv^g
+\int_{S^n}uP_0u\,dv^g,
\end{displaymath}
where $\eta$ is a smooth function supported in 
$B(q,\delta)$. \\

\noindent Now, let $f$: $M\to\mR$ be a smooth 
function which is positive on 
$M\setminus \overline{B(p, \delta)}$ and $0$ 
on $B(p,\de)$. 
We consider the operator 
$P_a := \Delta_ g + f + a f$. 
Let $m(a)$ be the corresponding mass. 
For every $a\geq 0$ the operator $P_a$ is positive. 
By Point 4 of 
Theorem \ref{massfunctioncor}, we have 
\begin{displaymath}
\lim_{a + \infty} m(a) = 
- \inf\{J_0(u)|\, u\in C^{\infty}(M),\, 
\supp(u)\subset B(p,\delta) \}
\end{displaymath}
where $J_0$ is constructed as above on $M$. Observe that since $(B(q,\de),h) \subset (S^n,h)$ and $(B(p,\de),g) \subset (M,g)$ are isometric, 
we have 
\begin{align*} 
&\inf\{J_0(u)|\, u\in C^{\infty}(M),\, 
\supp(u)\subset B(p,\delta) \}  \\
& =\inf\{J(u)|\, u\in C^{\infty}(S^n),\, 
\supp(u)\subset B(q,\delta) \}.
\end{align*}

\noindent By (\ref{limonsphere}), we obtain that $\lim_{a \to \infty} m(a) < 0$ which proves Corollary \ref{massfunctioncor}.

\subsubsection{Proof of Corollary \ref{massfunctioncor2}}
It is sufficient to find a Riemannian metric $g$ 
on $S^n$ which is conformal to $\sigma^n$, 
flat on an open neighborhood 
of $p$ and satisfies $\scal_g\geq 0$. 
Choose an open neighborhood $U$ of $p$ on which 
$\sigma^n$ is conformally flat. 
Using stereographic projection at $-p$ we may 
write $\sigma^n=u^{4/(n-2)}\xi^n$ on $U$ where 
with $r=|x|_{\xi^n}$ we have 
\begin{displaymath}
u(r)=\Big(\frac{2}{1+r^2}\Big)^{\frac{n-2}{2}}.
\end{displaymath}
Let $\varepsilon>0$ be so small that 
$u''(r)<0$ on $[0,2\varepsilon)$ and such that 
the preimage of $B(0,2\varepsilon)\subset\mR^n$ 
under the stereographic projection is contained 
in $U$. 
Choose a smooth function $v$ on $[0,\infty)$ 
such that $v$ is constant on $[0,\varepsilon)$, 
$v=u$ on $[2\varepsilon,\infty)$ and such that 
on $[0,2\varepsilon)$ we have 
$v'(r)\leq 0$ and $v''(r)\leq 0$. 
We define $v$ as a radial function on $\mR^n$ 
and we obtain 
\begin{equation}
\label{Delta_xi_v}
\Delta_{\xi^n}v(x)=-v''(r)-\frac{n-1}{r}v'(r)\geq 0.
\end{equation}
We define $g=v^{4/(n-2)}\xi^n$ on $U$ and 
$g=\sigma^n$ on $S^n\setminus U$. 
Then $g$ is 
a smooth Riemannian metric on $S^n$ which is 
conformal to $\sigma^n$ and flat on an open 
neighborhood of $p$. 
Furthermore by the conformal transformation law 
(\ref{scalgg'}) for $L_{\xi^n}$ and by 
(\ref{Delta_xi_v}) we have $\scal_g\geq 0$ 
on $S^n$.

\subsection{Application 4: Surgery and positivity of mass} 
\label{section_surgery_positivity}

Let $(M,g)$ be a closed Riemannian manifold of dimension $n \geq 3$, let $p\in M$ and assume that $g$ is flat on an  open neighborhood $U$ of $p$. 
Let $f\in C^{\infty}(M)$ such that $f\equiv 0$ on $U$.
We keep the same notation as in Section \ref{variational}. Let now $\Om\subset M$ be an open subset containing $\supp(\eta)$. 
Assume that 
\begin{displaymath}
P_f|_{\Om}:\quad C^{\infty}(\Omega)
\to C^{\infty}(\Omega)
\end{displaymath}
is a positive operator with respect to 
Dirichlet boundary condition. 
Then, we define
\begin{displaymath}
m_{f, \Om} := 
-\inf\{I_f(u)|\,u\in C^{\infty}(M),\, u(p)=0,\, 
\supp(u) \subset \Omega\}
\end{displaymath}
Let $G_{f,\Om}$ be the Green function of $P_f|_{\Om}$ with Dirichlet boundary condition. Mimicking the proof of Theorem \ref{maintheorem}, one proves that $m_{f, \Om}$ is the mass of $G_{f,\Om}$. 
Clearly for any $\Om$ the following proposition is obvious from the definitions. 
\begin{proposition} \label{massom}
We have 
\begin{displaymath}
m_f \geq m_{f,\Om}.
\end{displaymath}
\end{proposition}
This observation has nevertheless some interesting applications. A first one is the following: let $(\Omega,g_0)$ be a compact manifold with boundary and let $f_0$ be a function defined on $\Om$.  Assume that $(\Om, g_0)$ embeds isometrically in $(M,g)$  and let $f$ be such that $P_f$ is positive on $(M,g)$ and $f = f_0$ on $\Om \subset M$. Then, the mass of $P_f$ is bounded from below by a constant which depends only on $(\Om,g_0)$ and $f_0$. \\

\noindent Another application seems much more interesting. 
Let $(M,g)$ be a closed Riemannian manifold with positive Yamabe constant $Y(M,g)$. 
We assume that $g$ is flat around a point $p$. 
Now, we perform on $M$ a surgery of dimension 
$k \leq n-3$, i.\,e.\,we remove a tubular 
neighborhood of a sphere $S^k$ in $M$ and replace 
it by gluing the boundary with the boundary 
of the product 
$\overline{B}^{k+1} \times S^{n-k-1}$. 
Without loss of generality, we can assume that 
$p$ does not lie in the removed part. 
For more information on this procedure, see for 
instance \cite{ammann.dahl.humbert:13} 
or Section \ref{section_definition_metrics}. 
Then, it was proven by several authors (see 
\cite{gromov.lawson:80, schoen.yau:79c,  ammann.dahl.humbert:13}) 
that on the new manifold $N$ one can construct a 
new metric $h$ with positive Yamabe constant 
which is flat  around $p$. 
Moreover $h$ can be constructed in such a way that 
it coincides with $g$ on $M$ 
except on an arbitrarily small open neighborhood of 
the removed sphere in $M$. 
Then, a natural question is: assume that the mass 
$m(M,g)$ of $L_g$ is positive. Does this imply that 
the mass  $m(N,h)$ of $L_h$ is also positive? 
Observe that  Proposition \ref{massom} gives an 
immediate positive answer to this question. 
Indeed, for $\ep>0$, define 
\begin{displaymath}
\Om_\ep:=\{ x \in M |\; \dist_g(x, S) > \ep \}
\end{displaymath} 
where $S$ is the surgery $k$-sphere. 
Then we prove the following theorem. 

\begin{theorem}
\label{theorem_surgery_positivity}
For every $\varepsilon>0$ let $h_{\varepsilon}$ 
be a Riemannian metric on $N$ such that 
$Y(N,h_{\varepsilon})>0$ and $h_{\varepsilon}=g$ 
on $\Omega_{\varepsilon}$. 
Then we have 
\begin{displaymath}
\liminf_{\varepsilon\to 0}m(N,h_{\varepsilon})
\geq m(M,g). 
\end{displaymath}
\end{theorem}

\begin{proof}
Let $u\in C^{\infty}(M)$ such that we have 
$-m(M,g)=J_g(u)$. 
Let $\chi_\ep$ be a smooth function on $M$ 
equal to $1$ 
on $\Om_{2 \ep}$, equal to $0$ 
on $M \setminus \Om_\ep$ and such that 
$|\diff \chi_\ep |_g \leq \frac{2}{\ep}$. 
We may consider the functions $\chi_{\varepsilon}u$ 
as functions on $N$. 
We write 
$A_{\varepsilon}:=\Omega_{\varepsilon}\setminus
\Omega_{2\varepsilon}$. 
Since on $\supp(\chi_{\varepsilon})$ we have 
$h_{\varepsilon}=g$ we obtain by 
(\ref{cut_off_formula}) 
\begin{align*}
\int_{N} u\chi_{\varepsilon}L_{h_{\varepsilon}}
(u\chi_{\varepsilon})\,dv^{h_{\varepsilon}}&=
\int_{M} 
(u^2|\diff\chi_{\varepsilon}|_g^2
+\chi_{\varepsilon}^2uL_g u)\,dv^g\\
&\leq \frac{4}{\varepsilon^2}
\int_{A_{\varepsilon}}u^2\,dv^g
+\int_{M}\chi_{\varepsilon}^2uL_gu\,dv^g.
\end{align*}
Let $k\in\{0,...,n-3\}$ be the dimension of the 
surgery sphere. 
Since there exists $C>0$ such that 
for all $\varepsilon$ we have 
$\vol(A_{\varepsilon})\leq C\varepsilon^{n-k}$, 
the first term on the 
right hand side tends to $0$ as 
$\varepsilon\to 0$. 
We conclude that 
\begin{displaymath}
\limsup_{\varepsilon\to 0}
J_{h_{\varepsilon}}(u\chi_{\varepsilon})
\leq J_g(u).
\end{displaymath}
Since $\supp(\chi_{\varepsilon}u)\subset
\Omega_{\varepsilon}$ it follows that  
\begin{displaymath}
\limsup_{\varepsilon\to 0}
\big(-m_{\frac{n-2}{4(n-1)}\scal_{h_{\varepsilon}},\Om_{\varepsilon}}
\big)
\leq J_g(u).
\end{displaymath}
and thus 
\begin{displaymath}
m(M,g)\leq
\liminf_{\varepsilon\to 0} 
m_{\frac{n-2}{4(n-1)}\scal_{h_{\varepsilon}},\Omega_{\varepsilon}}.
\end{displaymath}
The assertion now follows from Proposition 
\ref{massom}. 
\end{proof}


\noindent Theorem \ref{theorem_surgery_positivity} shows that the positivity of mass is 
preserved by surgery of dimension 
$k\in\{0,...,n-3\}$. 
In the next section we will obtain a much stronger 
result, namely that also a negative mass is 
preserved under such surgeries.

\section{Preservation of mass by surgery}
\label{section_surgery_preservation}

\subsection{The result}

Let $(M,g)$ be a closed 
Riemannian manifold of dimension $n\geq3$ 
with positive Yamabe constant $Y(M,g)$. 
Assume that $g$ is flat on an open 
neighborhood of a point $p\in M$. 
Then we can define the mass $m(M,g)$ at $p$. 
Let $N$ be obtained from $M$ by a surgery 
of dimension $k\in\{0,...,n-3\}$ which does 
not hit the point $p$. 
Our aim is to show that the mass $m(M,g)$ at $p$ 
is preserved by this procedure. 
More precisely we will prove the following 
theorem. 

\begin{theorem}
\label{theorem_preservation_mass}
There exists a sequence of metrics 
$(g_{\theta})$ on $N$ 
such that for every $\theta$ the mass 
$m(N,g_{\theta})$ at $p$ can be defined 
and such that we have 
\begin{displaymath}
\lim_{\theta\to 0}m(N,g_{\theta})=m(M,g). 
\end{displaymath}
\end{theorem}

We will study an application of this theorem 
to the positive mass conjecture 
in Section \ref{section_pmt_non_spin}. 
But first we will prove Theorem 
\ref{theorem_preservation_mass}. 
We will define the family of metrics $g_{\theta}$ 
in Section \ref{section_definition_metrics}. 
The same family of metrics has been used in 
the article \cite{ammann.dahl.humbert:13}. 
In Section 
\ref{section_proof_preservation_mass} we will 
prove that this family of metrics has the 
property stated in the theorem. 
We will use the variational characterization 
of the mass according to Theorem \ref{maintheorem} 
and we will also use some techniques from 
the article \cite{ammann.dahl.humbert:13}, which 
we briefly recall in 
Sections \ref{section_limit_spaces} and 
\ref{section_ws_bundles}.

\subsection{Definition of the metrics $g_{\theta}$}
\label{section_definition_metrics}

We recall a construction called the connected sum 
along a submanifold using the notation of the article 
\cite{ammann.dahl.humbert:13}. 
On the manifold obtained in this way we 
define a family of Riemannian metrics 
$(g_{\theta})_{\theta>0}$ which is described in 
the same article. 
We will mostly be interested in surgery which is 
a special case of this construction. 
Let $(M_1,g_1)$, $(M_2,g_2)$ be complete Riemannian 
manifolds of dimension $n$ and let $W$ be a 
closed manifold of dimension $k\leq n$. 
Let $\bar{w}_i$: $W\times\mR^{n-k}\to TM_i$, $i=1,2$, 
be embeddings. 
We assume that $\bar{w}_i$ 
maps $W\times\{0\}$ to the zero 
section of $TM_i$ which we identify with $M_i$. 
Thus we obtain embeddings $W\to M_i$ and we will 
denote the images of these embeddings by 
$W_i'\subset M_i$. 
We assume that for every $x\in W$ the 
embeddings $\bar{w}_i$ restrict to linear 
isomorphisms 
$\{x\}\times\mR^{n-k}\to N_{\bar{w}_i(x,0)}W_i'$, where 
$NW_i'$ denotes the normal bundle of $W_i'$ with respect 
to the metric $g_i$. 
For $i=1,2$ let $r_i$ be the function on $M_i$ 
giving the distance to $W_i'$ and define 
$U^{M_i}(c):=\{x\in M_i|\,r_i(x)<c\}$ for every $c>0$. 
There exists $R_{\max}>0$ such that 
the maps $w_i:=\exp^{g_i}\circ\bar{w}_i$ define 
diffeomorphisms 
\begin{displaymath}
w_i:\quad W\times B^{n-k}(R_{\max})\to 
U^{M_i}(R_{\max}), 
\quad i=1,2.
\end{displaymath}
In general, the Riemannian metrics $g_i$ do not have 
a corresponding product structure on 
$U^{M_i}(R_{\max})$. 
We introduce error terms $T_i$ measuring the 
differences 
from the product metrics. 
Namely, if $h_i$ denote the restrictions of $g_i$ 
to $W_i'$ and if $\sigma^{n-k-1}$ is the standard 
metric on $S^{n-k-1}$ we have 
\begin{displaymath}
g_i=h_i+dr_i^2+r_i^2\sigma^{n-k-1}+T_i
\end{displaymath}
on $U^{M_i}(R_{\max})$, $i=1,2$. 
Now, for every $\epsilon\in(0,R_{\max})$ we define 
\begin{displaymath}
N_{\epsilon}:=(M_1\setminus U^{M_1}(\epsilon))\cup 
(M_2\setminus U^{M_2}(\epsilon))/\sim,
\end{displaymath}
and for every $c\in(\epsilon,R_{\max})$ 
\begin{displaymath}
U^{N}_{\epsilon}(c):=
(U^{M_1}(c)\setminus U^{M_1}(\epsilon))\cup 
(U^{M_2}(c)\setminus U^{M_2}(\epsilon))/\sim,
\end{displaymath}
where $\sim$ means that we identify the point 
$x\in\partial U^{M_1}(\epsilon)$ with the point 
$w_2\circ w_1^{-1}(x)\in\partial U^{M_2}(\epsilon)$. 
Therefore we have 
\begin{displaymath}
N_{\epsilon}=(M_1\setminus U^{M_1}(c))\cup 
(M_2\setminus U^{M_2}(c))\cup 
U^{N}_{\epsilon}(c).
\end{displaymath}
We say that $N_{\epsilon}$ is obtained from $M_1$ 
and $M_2$ by a connected sum along $W$ with 
parameter $\epsilon$. 
Since the diffeomorphism type of the manifold 
$N_{\epsilon}$ is independent of the choice of 
$\epsilon$ we will often write $N$ instead of 
$N_{\epsilon}$. 
Our next aim is to define for a given $\theta>0$ 
a Riemannian metric $g_{\theta}$ 
on $N_{\epsilon}$ for $\epsilon>0$ small enough. 
We choose numbers $R_0$, $\theta$, $\delta_0$ 
such that 
\begin{displaymath}
R_{\max}>R_0>\theta>\delta_0>0.
\end{displaymath}
Then we choose 
$A_{\theta}\in(\theta^{-1},(\delta_0)^{-1})$ 
and we put $\epsilon:=e^{-A_{\theta}}\delta_0$. 
Then we define $N_{\epsilon}$ and 
$U^{N}_{\epsilon}(c)$ for $c>0$ as above.  
On the set $U^{N}_{\epsilon}(R_{\max})$ we define 
the coordinate function $t$ by 
\begin{displaymath}
t:=\left\{
\begin{array}{ll}
-\ln r_1+\ln\epsilon,&
\textrm{on }U^{M_1}(R_{\max})
\setminus U^{M_1}(\epsilon),\\
\ln r_2-\ln\epsilon,&
\textrm{on }U^{M_2}(R_{\max})
\setminus U^{M_2}(\epsilon).
\end{array} 
\right.
\end{displaymath}
We choose smooth functions $F$ on $N_{\epsilon}$ 
and $f$ on $U^{N}_{\epsilon}(R_{\max})$ such that 
\begin{align*}
F(x)&=
\left\{
\begin{array}{ll}
1, & \textrm{if }
x\in N_{\epsilon}\setminus U^{N}_{\epsilon}(R_{\max}),\\
r_i(x)^{-1}, & \textrm{if }
x\in U^{M_i}(R_0)\setminus U^{M_i}(\epsilon),\,i=1,2,
\end{array} 
\right.\\
f(x)&=
\left\{
\begin{array}{ll}
-|t(x)|-\ln\epsilon, & \textrm{if }
x\in N_{\epsilon}\setminus U^{N}_{\epsilon}(\theta),\\
\ln A_{\theta}, & \textrm{if }
x\in U^{N}_{\epsilon}(\delta_0)
\end{array} 
\right.
\end{align*}
and such that $|df/dt|\leq 1$ for all $t$ and 
$\|d^2f/dt^2\|_{L^{\infty}}\to 0$ as $\theta\to 0$. 
We choose a smooth function $\chi$: $\mR\to[0,1]$ 
such that $\chi=0$ on $(-\infty,-1]$, 
$\chi=1$ on $[1,\infty)$ and $|\chi'|\leq 1$ 
on $\mR$. 
Then we define 
\begin{displaymath}
g_{\theta}:=
\left\{
\begin{array}{ll}
F^2g_i, & 
\textrm{on }M_i\setminus U^{M_i}(\theta),\\
e^{2f(t)}(h_i+T_i)+dt^2+\sigma^{n-k-1}, & 
\textrm{on }U^{M_i}(\theta)
\setminus U^{M_i}(\delta_0),\\
\hspace{-0.7em}\left. 
\begin{array}{l}
A_{\theta}^2\chi(t/A_{\theta})(h_2+T_2)\\
+A_{\theta}^2(1-\chi(t/A_{\theta}))(h_1+T_1)\\
+dt^2+\sigma^{n-k-1}, 
\end{array}
\right\}
&
\textrm{on }U^{N}_{\epsilon}(\delta_0).
\end{array} 
\right.
\end{displaymath}
On $U^{N}_{\epsilon}(R_0)$ we write the metric 
$g_{\theta}$ as 
\begin{displaymath}
g_{\theta}=e^{2f(t)}\tilde{h}_t+dt^2
+\sigma^{n-k-1}+\widetilde{T}_t,
\end{displaymath}
where $\tilde{h}_t$ is defined by 
\begin{displaymath}
\tilde{h}_t:=
\chi(t/A_{\theta})h_2+(1-\chi(t/A_{\theta}))h_1,
\end{displaymath}
for $t\in\mR$ and where the error term 
$\widetilde{T}_t$ is equal to 
\begin{displaymath}
\widetilde{T}_t:=e^{2f(t)}
(\chi(t/A_{\theta})T_2+(1-\chi(t/A_{\theta}))T_1).
\end{displaymath}
On $U^{N}_{\epsilon}(R_0)$ we also define the 
metric without error term 
\begin{equation}
\label{g_theta_primed}
g_{\theta}':=g_{\theta}-\widetilde{T}_t
=e^{2f(t)}\tilde{h}_t+dt^2+\sigma^{n-k-1}.
\end{equation}
We will need upper bounds for the error term 
$\widetilde{T}$ and its derivatives. 
As in Section 6.2 of the article 
\cite{ammann.dahl.humbert:13} one can show that 
there exists $C>0$ such that for all $\theta$ 
we have 
\begin{align}
\label{T_estimate_1}
|\widetilde{T}_t|_{g_{\theta}'}
&\leq Ce^{-f(t)}\\
\label{T_estimate_2}
|\nabla^{g_{\theta}'}\widetilde{T}_t|_{g_{\theta}'}
&\leq Ce^{-f(t)}\\
\label{T_estimate_3}
|\scal_{g_{\theta}}-\scal_{g_{\theta}'}|
&\leq Ce^{-f(t)}.
\end{align}
In the special case where $M_2=S^n$, $W=S^k$, 
$k\leq n$, and $S^k\to S^n$ is the standard embedding 
we say that $N_{\epsilon}$ is obtained from $M_1$ 
by surgery of dimension $k$ with parameter $\epsilon$. 
Note that in this case 
$M_2\setminus U^{M_2}(\epsilon)$ is diffeomorphic to 
$\overline{B}^{k+1}\times S^{n-k-1}$.

\subsection{Limit spaces and limit solutions}
\label{section_limit_spaces}

In the proof of Theorem 
\ref{theorem_preservation_mass} we will construct 
solutions to the equation $\Delta_g u=0$ on 
certain limit spaces $(V,g)$. 
We need the following lemmas which are 
adapted versions of Lemmas 4.1, 4.2 and 4.3 
in \cite{ammann.dahl.humbert:13}. 

\begin{lemma}
\label{lemma_limit_space_1}
Let $V$ be a manifold of dimension $n$. 
Let $(q_{\alpha})_{\alpha}$ 
be a sequence of points in $V$ 
that converges to a point $q$ as $\alpha\to 0$. 
Let $(\gamma_{\alpha})_{\alpha}$ 
be a sequence of metrics defined on an open 
neighborhood $O$ of $q$ that converges to 
a metric $\gamma_0$ in the $C^2(O)$-topology 
as $\alpha\to 0$. 
Let $(b_{\alpha})_{\alpha}$ be 
a sequence of positive real numbers such that 
$b_{\alpha}\to\infty$ as $\alpha\to 0$. 
Then for every $r>0$ there exists for 
$\alpha$ small enough a diffeomorphism 
\begin{displaymath}
\Theta_{\alpha}:\quad 
B^n(r)\to 
B^{\gamma_{\alpha}}(q_{\alpha},b_{\alpha}^{-1}r)
\end{displaymath}
with $\Theta_{\alpha}(0)=q_{\alpha}$ such that 
the metric 
$\Theta_{\alpha}^*(b_{\alpha}^2\gamma_{\alpha})$ 
tends to the flat metric $\xi^n$ in $C^2(B^n(r))$. 
\end{lemma}

\begin{proof}
see the proof of Lemma 4.1 
in \cite{ammann.dahl.humbert:13}. 
\end{proof}

\begin{lemma}
\label{lemma_limit_space_2}
Let $V$ be a manifold of dimension $n$. 
Let $(g_{\alpha})_{\alpha}$ be a sequence of 
metrics that 
converges to a metric $g$ in $C^2$ on all compact 
sets $K\subset V$ as $\alpha\to 0$. 
Assume that $(U_{\alpha})_{\alpha}$ is an increasing 
sequence of subdomains of $V$ such that 
$\bigcup_{\alpha}U_{\alpha}=V$. 
Let $u_{\alpha}\in C^2(U_{\alpha})$ be a sequence 
of positive functions such that 
$\|u_{\alpha}\|_{L^{\infty}(U_{\alpha})}$ is 
bounded independently of $\alpha$. 
We assume 
\begin{displaymath}
L_{g_{\alpha}}u_{\alpha}=0
\end{displaymath}
for all $\alpha$. 
Then there exists a non-negative function 
$u\in C^2(V)$ satisfying 
\begin{displaymath}
L_gu=0
\end{displaymath}
on $V$ and a subsequence of $u_{\alpha}$ that tends 
to $u$ in $C^1$ on each open set $\Omega\subset V$ 
with compact closure. 
In particular for every compact subset $K\subset V$ 
we have 
\begin{equation}
\label{limit_L_infty_norm}
\|u\|_{L^{\infty}(K)}=
\lim_{\alpha\to 0}\|\,u_{\alpha}\|_{L^{\infty}(K)}
\end{equation}
and 
\begin{equation}
\label{limit_L_r_norm}
\int_K u^r\,dv^g
=\lim_{\alpha\to 0}
\int_K u_{\alpha}^r\,dv^{g_{\alpha}}
\end{equation}
for every $r\geq 1$. 
\end{lemma}

\begin{proof}
see the proof of Lemma 4.2  
in \cite{ammann.dahl.humbert:13}.
\end{proof} 

\begin{lemma}
\label{lemma_limit_space_3}
Let $\xi^n$ be the flat metric on $\mR^n$ and assume 
that $u\in C^2(\mR^n)$, $u\geq 0$, $u\not\equiv 0$ 
satisfies 
\begin{displaymath}
L_{\xi^n}u=\mu u^{p-1}
\end{displaymath}
for some $\mu\in\mR$ and $p:=\frac{2n}{n-2}$. 
Assume in addition that $u\in L^p(\mR^n)$ and that 
\begin{displaymath}
\|u\|_{L^p(\mR^n)}\leq 1.
\end{displaymath}
Then $\mu\geq Y(S^n,\sigma^n)$. 
\end{lemma}

\begin{proof}
see the proof of Lemma 4.3 in 
\cite{ammann.dahl.humbert:13}. 
\end{proof}

\subsection{$L^2$-estimates on $WS$-bundles}
\label{section_ws_bundles}

\begin{definition}
Let $n\geq 1$ and $k\in\{0,...,n-3\}$ be integers. 
Let $W$ be a closed manifold of dimension $k$ 
and let $I$ be an interval. 
A $WS$-bundle is a product 
$P:=I\times W\times S^{n-k-1}$ equipped with a 
metric of the form 
\begin{equation}
\label{g_WS}
g_{\mathrm{WS}}=dt^2+e^{2\varphi(t)}h_t
+\sigma^{n-k-1}
\end{equation}
where $h_t$ is a smooth family of metrics on $W$ 
depending on $t\in I$ and $\varphi$ is a function 
on $I$. 
\end{definition}
We denote by $\pi$: $P\to I$ the projection onto 
the first factor and for every $t\in I$ 
we write $F_t:=\pi^{-1}(t)$. 
Furthermore we define 
\begin{displaymath}
e(h_t):=\frac{1}{2(n-1)}
\mathrm{tr}_{h_t}(\partial_t h_t).
\end{displaymath}
\begin{definition}
We say that condition $(A_t)$ holds at $t\in I$, 
if the following assumptions are true:
\begin{enumerate}[1.]
\item $s\mapsto h_s$ is constant on 
an open neighborhood of $t$,
\item $e^{-2\varphi(t)}\inf_{x\in W}\scal_{h_t}(x)
\geq-\frac{(n-k-2)(n-1)}{8(n-2)}$,
\item $|\varphi'(t)|\leq 1$
\item $0\leq -2k\varphi''(t)\leq
\frac{1}{2}(n-1)(n-k-2)^2$.
\end{enumerate}
We say that condition $(B_t)$ holds at $t\in I$, 
if the following assumptions are true:
\begin{enumerate}[1.]
\item $s\mapsto\varphi(s)$ is constant on 
an open neighborhood of $t$,
\item $\inf_{x\in F_t}\scal_{g_{\mathrm{WS}}}(x)
\geq\frac{1}{2}\scal_{\sigma^{n-k-1}}
=\frac{1}{2}(n-k-1)(n-k-2)$,
\item $\frac{(n-1)^2}{2}e(h_t)^2
+\frac{n-1}{2}\partial_t e(h_t)
\geq-\frac{3}{64}(n-k-2)$.
\end{enumerate}
\end{definition}
Let $P$ be a $WS$-bundle and let $G$ be a 
Riemannian metric on $P$ which is close 
to $g_{\mathrm{WS}}$ in a sense we will make 
precise later. 
Assume that $u$ satisfies the equation 
\begin{equation}
\label{LGu=0}
L_{G}u=0.
\end{equation}
Our aim is to estimate the distribution of 
$L^2$-norm of $u$ with respect to the metric 
$g_{\mathrm{WS}}$. 
If we rewrite the equation (\ref{LGu=0}) 
in terms of the metric $g_{\mathrm{WS}}$ 
we obtain an equation of the form 
\begin{equation}
\label{perturbed_LGu=0}
L_{g_{\mathrm{WS}}}u
=d^*A(du)+Xu+\epsilon\partial_t u -su,
\end{equation}
where $s,\epsilon\in C^{\infty}(P)$, 
$A\in\Gamma(\mathrm{End}(T^*P))$ and 
$X\in\Gamma(TP)$ and where $dt(X)=0$ 
and $A(dt)=0$ and $A$ is symmetric. 
Then the following theorem holds. 

\begin{theorem}
\label{theorem_ws_bundle}
Assume that $P$ is equipped with a metric 
$g_{\mathrm{WS}}$ of the form (\ref{g_WS}). 
Let $\alpha,\beta\in\mR$ such that 
$[\alpha,\beta]\subset I$. 
Assume that for every $t\in I$ condition $(A_t)$ 
or condition $(B_t)$ holds. 
Assume that $u$ is a positive solution 
of (\ref{perturbed_LGu=0}). 
Then there exists $c_0>0$ independent of 
$\alpha,\beta$ and $\varphi$ such that if 
\begin{displaymath}
\|A\|_{L^{\infty}(P)},
\|X\|_{L^{\infty}(P)},
\|s\|_{L^{\infty}(P)},
\|\epsilon\|_{L^{\infty}(P)},
\|e(h_t)\|_{L^{\infty}(P)}
\leq c_0,
\end{displaymath}
then 
\begin{displaymath}
\int_{\pi^{-1}((\alpha+\gamma,\beta-\gamma))}
u^2\,dv^{g_{\mathrm{WS}}}
\leq 
\frac{4(\vol^{g_\alpha}(F_{\alpha})
+\vol^{g_\beta}(F^{\beta}))}{n-k-2}
\|u\|_{L^{\infty}(\pi^{-1}(\alpha,\beta))}^2,
\end{displaymath}
where $\gamma:=\frac{\sqrt{32}}{n-k-2}$. 
\end{theorem}
Note that the assertion is non-trivial only if 
$\beta-\alpha>2\gamma$. 

\begin{proof}
This is a special case of Theorem 5.2 
in \cite{ammann.dahl.humbert:13}. 
Since the proof given there is very long 
and technical we will not repeat it here. 
Note that the theorem 
in \cite{ammann.dahl.humbert:13} is stated with 
$\|u\|_{L^{\infty}(P)}$ on the right hand side 
of the asserted estimate. 
However if we examine the end of the proof 
of Theorem 5.2 
in \cite{ammann.dahl.humbert:13} we observe 
that we may also put 
$\|u\|_{L^{\infty}(\pi^{-1}(\alpha,\beta))}$ 
as we have done. 
\end{proof}

\subsection{Proof of Theorem \ref{theorem_preservation_mass}}
\label{section_proof_preservation_mass}

Let $(M,g)$ be a closed 
Riemannian manifold of dimension $n\geq3$ 
with positive Yamabe constant $Y(M,g)$. 
Assume that $g$ is flat on an open 
neighborhood $U$ of a point $p\in M$. 
Then we can define the mass $m(M,g)$ at $p$. 

Let $N$ be obtained from $M$ by a surgery 
of dimension $k\in\{0,...,n-3\}$ which does not 
hit the point $p$. 
More precisely we apply the construction described 
in Section \ref{section_definition_metrics} with 
$M_1:=M$, $g_1:=g$, $M_2:=S^n$, $g_2:=\sigma^n$, 
$W:=S^k$ such that the 
embedding $S^k\to S^n$ is the standard embedding 
and such that $p$ is not contained in the image of the 
embedding $S^k\to M$. 
Moreover we choose the number $R_{\max}>0$ and the 
open neighborhood $U$ of $p$ in such a way that 
$U\cap U^{M}(R_{\max})=\emptyset$. 
Then for all $\theta$ which are small enough 
we obtain a manifold $N:=N_{\epsilon}$ with a 
Riemannian metric 
$g_{\theta}$ as described in Section 
\ref{section_definition_metrics}. 
In particular $g_{\theta}$ coincides 
with $g$ on $U$. 
By Theorem 6.1 in the article 
\cite{ammann.dahl.humbert:13} 
and by the fact that 
\begin{displaymath}
Y(M\amalg S^n,g\amalg \sigma^n)
=Y(M,g)
\end{displaymath}
(see e.\,g.\,Section 1.2 in 
\cite{ammann.dahl.humbert:13}) 
we know that there exist positive 
constants $\Lambda_{n,k}$ depending only on $n$ 
and $k$, 
such that 
\begin{displaymath}
\min\{Y(M,g),\Lambda_{n,k}\}
\leq\liminf_{\theta\searrow 0}Y(N,g_{\theta})
\leq\limsup_{\theta\searrow 0}Y(N,g_{\theta})
\leq Y(M,g).
\end{displaymath}
Thus if $\theta$ is small enough we have 
$Y(N,g_{\theta})>0$ 
and thus we can define the mass $m(N,g_{\theta})$ 
at $p$.

We recall that by Theorem \ref{maintheorem} we have 
\begin{displaymath}
-m(M,g)
=\inf 
\{J_g(u)|\,u\in C^{\infty}(M)\},
\end{displaymath} 
where for every $u\in C^{\infty}(M)$ 
\begin{displaymath}
J_g(u)
=\int_{M\setminus\{p\}}
\eta r^{2-n}F_{\eta}\,dv^g
+2\int_{M}
uF_{\eta}\,dv^g
+\int_{M}uL_g u\,dv^g
\end{displaymath} 
and where $\eta$ and $F_{\eta}$ are defined 
as in Section \ref{variational}. 
For $m(N,g_{\theta})$ we have an analogous 
formula with a functional denoted 
by $J_{g_{\theta}}$. 
We note that the functions $\eta$ and $F_{\eta}$ 
can be chosen independently of $\theta$ since 
we have $g=g_{\theta}$ for all $\theta$ 
on $\supp(\eta)$. 


The proof of Theorem \ref{theorem_preservation_mass} 
is divided into several steps.\\

\noindent 
\textbf{Step 1:} 
After passing to a subsequence we have 
\begin{displaymath}
\lim_{\theta\to 0} m(N,g_{\theta})\geq m(M,g). 
\end{displaymath}

The proof is analogous to the proof of Theorem 
\ref{theorem_surgery_positivity} and we do not 
repeat it here. 


We choose $\delta>0$ such that 
$B(p,2\delta)\subset U$ and we choose 
a smooth function $\eta$ on $N_{\epsilon}$ 
such that $\eta\equiv\frac{1}{(n-2)\omega_{n-1}}$ 
on $B(p,\delta)$, 
$\eta\equiv 0$ on 
$N_{\epsilon}\setminus B(p,2\delta)$ and 
$|\diff\eta|_g\leq\frac{2}{\delta}$ on 
$N_{\epsilon}$. 
For every $\theta$ we denote the Green function 
for $L_{g_{\theta}}$ at $p$ by $G_{\theta}$. 
Then the function $u_{\theta}$: 
$N_{\epsilon}\to\mR$, 
\begin{displaymath}
u_{\theta}(x):=
\left\{
\begin{array}{ll}
G_{\theta}(x)-\eta(x) r(x)^{2-n}, & x\neq p\\
m(N,g_{\theta}), & x=p
\end{array}
\right. 
\end{displaymath}
is smooth. 
For every $\alpha>0$ which is small enough we set 
\begin{displaymath}
A_{\alpha}:=
U^M(2\alpha)\setminus U^M(\alpha)
\subset M.
\end{displaymath} 

\noindent 
\textbf{Step 2:} 
We prove that 
for all $\alpha,\theta$ with 
$0<\theta<\alpha< R_0$ we have 
\begin{displaymath}
-m(M,g)\leq 
-m(N,g_{\theta})
+16\int_{A_{\alpha}} u_{\theta}^2\,dv^{g_{\theta}}.
\end{displaymath}

For every $\alpha$ which is small enough 
let $\chi_{\alpha}$: $M\to[0,1]$ 
be a smooth function such that 
$\chi_{\alpha}\equiv 1$ 
on $M\setminus U^M(2\alpha)$, 
$\chi_{\alpha}\equiv 0$ on 
$U^M(\alpha)$ 
and $|\diff\chi_{\alpha}|_g\leq\frac{2}{\alpha}$. 
In particular for all $\alpha$ 
we have $\chi_{\alpha}\equiv 1$ on $U$. 
Furthermore if $\theta<\alpha$, then we have 
$g_{\theta}=F^2g$ on $\supp(\chi_{\alpha})$. 
If in addition $\alpha\in(0,R_0)$, then 
we obtain for all $\theta\in(0,\alpha)$ 
\begin{equation}
\label{d_chi_alpha}
|\diff\chi_{\alpha}|_{g_{\theta}}
=F^{-1}|\diff\chi_{\alpha}|_g
=r|\diff\chi_{\alpha}|_g
\leq 2\alpha\frac{2}{\alpha}=4.
\end{equation}
For $0<\theta<\alpha<R_0$ the function 
$v_{\alpha,\theta}$: $M\to\mR$ defined by 
\begin{displaymath}
v_{\alpha,\theta}(x):=
\left\{
\begin{array}{ll}
F^{\frac{n-2}{2}}\chi_{\alpha} u_{\theta},&
\textrm{if }x\in M\setminus W_1'\\
0,& \textrm{if }x\in W_1'
\end{array} 
\right. 
\end{displaymath}
is smooth. 
By Theorem \ref{maintheorem} we have 
\begin{align*}
-m(M,g)
&\leq J_g(v_{\alpha,\theta})\\
&=\int_{M\setminus\{p\}}
\eta r^{2-n}F_{\eta}\,dv^g
+2\int_{M}
F^{\frac{n-2}{2}}\chi_{\alpha} u_{\theta}F_{\eta}\,
dv^g\\
&{}+\int_{M}F^{\frac{n-2}{2}}\chi_{\alpha} u_{\theta}
L_g(F^{\frac{n-2}{2}}\chi_{\alpha} u_{\theta})\,dv^g.
\end{align*}
Since on $\supp(\chi_{\alpha})$ we have 
$g_{\theta}=F^2g$ 
it follows from the conformal transformation property 
(\ref{L_g_conform}) of $L_g$ that 
\begin{displaymath}
L_g(F^{\frac{n-2}{2}}\chi_{\alpha} u_{\theta})=
F^{\frac{n+2}{2}}
L_{g_{\theta}}
(\chi_{\alpha} u_{\theta}). 
\end{displaymath}
Since on $\supp(\chi_{\alpha})$ we have 
$dv^g=F^{-n}dv^{g_{\theta}}$ we obtain 
\begin{displaymath}
\int_{M}
F^{\frac{n-2}{2}}\chi_{\alpha} u_{\theta}
L_g(F^{\frac{n-2}{2}}\chi_{\alpha} u_{\theta})\,
dv^g
=\int_{M} \chi_{\alpha} u_{\theta} 
L_{g_{\theta}} (\chi_{\alpha} u_{\theta})\, 
dv^{g_{\theta}}. 
\end{displaymath}
Now by (\ref{cut_off_formula}) we have 
\begin{displaymath}
\int_{M} \chi_{\alpha} u_{\theta} 
\Delta_{g_{\theta}} 
(\chi_{\alpha} u_{\theta})\,dv^{g_{\theta}}
=\int_{M} (u_{\theta}^2
|\diff\chi_{\alpha}|_{g_{\theta}}^2
+\chi_{\alpha}^2 
u_{\theta}\Delta_{g_{\theta}} u_{\theta})\,
dv^{g_{\theta}}.
\end{displaymath}
Using that 
on $\supp(F_{\eta})$ we have $F\equiv 1$, 
$\chi_{\alpha}\equiv 1$ 
and $g_{\theta}=g$ we obtain 
\begin{align*}
-m(M,g)&\leq
\int_{M\setminus\{p\}}
\eta r^{2-n}F_{\eta}\,dv^{g_{\theta}}
+2\int_{M}
u_{\theta}F_{\eta}\,dv^{g_{\theta}}\\
&{}
+\int_{M}
u_{\theta}^2
|\diff\chi_{\alpha}|_{g_{\theta}}^2\,
dv^{g_{\theta}}
+\int_{M}
\chi_{\alpha}^2 u_{\theta}
L_{g_{\theta}} u_{\theta}
\,dv^{g_{\theta}}.
\end{align*}
Using that $L_{g_{\theta}}u_{\theta}=-F_{\eta}$ 
and $\chi_{\alpha}\equiv 1$ 
on $\supp(F_{\eta})$ and using that 
$\supp(F_{\eta})\subset M$, we obtain 
\begin{align*}
-m(M,g)&\leq
\int_{N_{\epsilon}\setminus\{p\}}
\eta r^{2-n}F_{\eta}\,dv^{g_{\theta}}
+2\int_{N_{\epsilon}}
u_{\theta}F_{\eta}\,dv^{g_{\theta}}\\
&{}
+\int_{M}
u_{\theta}^2
|\diff\chi_{\alpha}|_{g_{\theta}}^2\,
dv^{g_{\theta}}
+\int_{N_{\epsilon}}
u_{\theta}
L_{g_{\theta}} u_{\theta}
\,dv^{g_{\theta}}\\
&=
J_{g_{\theta}}(u_{\theta})
+\int_{M} u_{\theta}^2
|\diff\chi_{\alpha}|_{g_{\theta}}^2\,
dv^{g_{\theta}}
\end{align*}
Using 
(\ref{d_chi_alpha}) and that 
$\supp(d\chi_\alpha)\subset A_{\alpha}$ 
we obtain 
\begin{displaymath}
-m(M,g)
\leq J_{g_{\theta}}(u_{\theta})
+16\int_{A_{\alpha}} u_{\theta}^2\,dv^{g_{\theta}}
\end{displaymath}
By Theorem \ref{maintheorem} we have 
$J_{g_{\theta}}(u_{\theta})
=-m(N,g_{\theta})$ and therefore the assertion 
of Step 2 follows.

In the remainder of the proof we will show
that the integral on the right 
hand side tends to $0$ as $\alpha$ and $\theta$ 
tend to $0$. 
By definition of $u_{\theta}$ we have 
$L_{g_{\theta}}u_{\theta}=-F_{\eta}$ 
for all $\theta$, where $F_{\eta}$ 
is defined as in Section \ref{variational}. 
In particular there exists $b>0$ such that 
for all $\theta$ we have 
\begin{displaymath}
U^N_{\epsilon}(b)\cap
\supp(L_{g_{\theta}}u_{\theta})
=\emptyset.
\end{displaymath}
In the following step we obtain an $L^2$-estimate 
for the functions $u_{\theta}$ which is independent 
of $\theta$. 
The result is not trivial since 
$\vol^{g_{\theta}}(U^{N}_{\epsilon}(b))
\to\infty$ as $\theta\to 0$.\\

\noindent 
\textbf{Step 3:} 
We prove that 
there exist $a\in(0,b)$ and $D>0$ such that 
for every $\theta$ we have 
\begin{displaymath}
\int_{U^N_{\epsilon}(a)} 
u_{\theta}^2\,dv^{g_{\theta}}
\leq 
D\Big(\max_{U^N_{\epsilon}(b)}u_{\theta}\Big)^2.
\end{displaymath}

This inequality is a special case of Lemma 6.6 in 
the article \cite{ammann.dahl.humbert:13} and we 
follow the proof given there. 
Let $\tilde{r}\in(\epsilon,b)$ be fixed. 
The manifold $P:=U^{N}_{\epsilon}(\tilde{r})$ 
with the metric $g_{\theta}'$ 
defined in (\ref{g_theta_primed}) is a 
$WS$-bundle, where in the notation of 
Section \ref{section_ws_bundles} we have 
$I=(\alpha,\beta)$ with 
$\alpha:=-\ln\tilde{r}+\ln\epsilon$ 
and $\beta:=\ln\tilde{r}-\ln\epsilon$. 
The metric $g_{\theta}'$ has exactly the form 
(\ref{g_WS}) with $\varphi=f$ and 
$h_t=\tilde{h}_t$. 
Let $\theta$ be small enough and let 
\begin{displaymath}
t\in(-\ln\tilde{r}+\ln\epsilon,
-\ln\delta_0+\ln\epsilon)\cup 
(\ln\delta_0-\ln\epsilon,
\ln\tilde{r}-\ln\epsilon).
\end{displaymath}
Then assumption $(A_t)$ from Section 
\ref{section_ws_bundles} is true. 
Let again $\theta$ be small enough and let 
\begin{displaymath}
t\in(-\ln\delta_0+\ln\epsilon,
\ln\delta_0-\ln\epsilon).
\end{displaymath}
Then we have 
$\scal_{g_{\theta}'}=\scal_{\sigma^{n-k-1}}
+O(1/A_{\theta})$ and the error term 
$e(\tilde{h}_t)$ from condition $(B_t)$ 
satisfies 
\begin{displaymath}
2(n-1)|e(\tilde{h}_t)|
\leq\Big|\mathrm{tr}_{\tilde{h}_t}
\partial_t\tilde{h}_t\Big|
=\Big|\mathrm{tr}_{\tilde{h}_t}
\Big(\chi'(t/A_{\theta})\frac{h_2-h_1}{A_{\theta}}
\Big)
\Big|
\leq\frac{C}{A_{\theta}}
\end{displaymath}
and 
\begin{displaymath}
2(n-1)|\partial_t e(\tilde{h}_t)|
=\big|\mathrm{tr}
\big(\tilde{h}_t^{-1}(\partial_t\tilde{h}_t)
\tilde{h}_t^{-1}(\partial_t\tilde{h}_t)
\big)
\big|
+\big|\mathrm{tr}_{\tilde{h}_t}
\partial_t^2\tilde{h}_t
\big|
\leq\frac{C}{A_{\theta}^2}.
\end{displaymath}
Because of $1/A_{\theta}\leq\theta$ 
the assumption $(B_t)$ from Section 
\ref{section_ws_bundles} is true. 
Now on $P$ we have $L_{g_{\theta}}u_{\theta}=0$ 
and with respect to the metric 
$g_{\mathrm{WS}}:=g_{\theta}'$ this equation 
has the form (\ref{perturbed_LGu=0}) 
as argued in Section \ref{section_ws_bundles}. 
Using (\ref{T_estimate_1}), (\ref{T_estimate_2}), 
(\ref{T_estimate_3}) one verifies that the 
error terms satisfy the pointwise estimates 
\begin{displaymath}
|A(x)|_{g_{\mathrm{WS}}},
|X(x)|_{g_{\mathrm{WS}}},
|s(x)|_{g_{\mathrm{WS}}},
|\epsilon(x)|_{g_{\mathrm{WS}}}
\leq Ce^{-f(t)}
\end{displaymath}
on $U^{N}_{\epsilon}(R_0)$, 
where $C>0$ is independent of $\theta$. 
In particular for every $c_0>0$ we obtain 
\begin{displaymath}
|A(x)|_{g_{\mathrm{WS}}},
|X(x)|_{g_{\mathrm{WS}}},
|s(x)|_{g_{\mathrm{WS}}},
|\epsilon(x)|_{g_{\mathrm{WS}}}
\leq c_0
\end{displaymath}
on $U^{N}_{\epsilon}(\theta)$ if $\theta$ 
is small enough. 
We set 
\begin{displaymath}
\alpha:=-\ln\tilde{r}+\ln\epsilon,\quad 
\beta:=\ln\tilde{r}-\ln\epsilon. 
\end{displaymath}
If $\tilde{r}$ is so small that 
$\beta-\alpha>2\gamma=\frac{8\sqrt{2}}{n-k-2}$, 
then with 
$P':=U^{N}_{\epsilon}(\tilde{r}e^{-\gamma})$ 
we obtain by Theorem \ref{theorem_ws_bundle} that 
\begin{displaymath}
\int_{P'}u_{\theta}^2\,dv^{g_{\mathrm{WS}}}
\leq C 
\|u_{\theta}\|_{L^{\infty}(\pi^{-1}(\alpha,\beta))}^2,
\end{displaymath}
where
\begin{displaymath}
C=\frac{4}{n-k-2}(\vol^{g_{\alpha}}(F^{\alpha})
+\vol^{g_{\beta}}(F^{\beta}))
\end{displaymath}
is independent of $\theta$. 
Furthermore if $\tilde{r}$ is small enough 
we have 
\begin{displaymath}
dv^{g_{\theta}}\leq 2dv^{g_{\mathrm{WS}}}
\end{displaymath}
on $P'$ and therefore 
\begin{displaymath}
\int_{P'}u_{\theta}^2\,dv^{g_{\theta}}
\leq 2 C 
\|u_{\theta}\|_{L^{\infty}(\pi^{-1}(\alpha,\beta))}^2.
\end{displaymath}
Thus with $a:=\tilde{r}e^{-\gamma}$ the 
assertion of Step 3 follows since the functions 
$u_{\theta}$ are positive on $U^N_{\epsilon}(b)$.\\

\noindent 
\textbf{Step 4:}
We prove that there exists $C_1>0$ such that 
for all $\theta$ we have 
\begin{displaymath}
\int_{N}u_{\theta}^p\,dv^{g_{\theta}}\leq C_1,
\end{displaymath}
where $p:=\frac{2n}{n-2}$. 

By Theorem 6.1 in the article 
\cite{ammann.dahl.humbert:13} there exists 
a positive constant $\Lambda_{n,k}$ depending 
only on $n$ and $k$ such that we have 
\begin{displaymath}
C_0:=\min\{Y(M,g),\Lambda_{n,k}\}
\leq\liminf_{\theta\to 0}Y(N,g_{\theta})
\end{displaymath}
where $C_0>0$. 
Let $q\in\mR$ such that 
$\frac{1}{p}+\frac{1}{q}=1$. 
By definition of $Y(N,g_{\theta})$ and by 
H\"older's inequality we obtain 
for all sufficiently small $\theta$ 
\begin{displaymath}
\frac{C_0}{2}\leq 
\frac{\int_{N}u_{\theta}L_{g_{\theta}}u_{\theta}\,
dv^{g_{\theta}}}
{(\int_{N}u_{\theta}^p\,dv^{g_{\theta}})^{2/p}}
=-\frac{\int_{N}u_{\theta}F_{\eta}\,dv^{g_{\theta}}}
{(\int_{N}u_{\theta}^p\,dv^{g_{\theta}})^{2/p}}
\leq\frac{(\int_{N}F_{\eta}^q\,dv^{g_{\theta}})
^{1/q}}
{(\int_{N}u_{\theta}^p\,dv^{g_{\theta}})^{1/p}}.
\end{displaymath}
On $\supp(F_{\eta})$ we have $g_{\theta}=g$ and 
thus the numerator on the right hand side is 
independent of $\theta$. 
The assertion of Step 4 follows.\\

\noindent 
\textbf{Step 5:}
We prove that there exists $C_2>0$ such that 
for all $\theta$ we have 
\begin{displaymath}
\max_{N_{\epsilon}}u_{\theta}\leq C_2.
\end{displaymath}

For every $\theta$ we choose 
$x_{\theta}\in N_{\epsilon}$ 
such that 
\begin{displaymath}
u_{\theta}(x_{\theta})
=\max_{N_{\epsilon}}u_{\theta}
=:m_{\theta}.
\end{displaymath}
We assume that after taking a subsequence we have 
$m_{\theta}\to\infty$ as $\theta\to 0$. 
First we prove the following lemma. 

\begin{lemma}
Let $\alpha>0$. 
Then for all sufficiently small $\theta$ 
there exists 
$x_{\theta}'\in U^{N}_{\epsilon}(2\alpha)$ 
such that we have 
\begin{displaymath}
c_{\alpha}m_{\theta}
\leq u_{\theta}(x_{\theta}')
\leq m_{\theta}, 
\end{displaymath}
where $c_{\alpha}>0$ is independent of $\theta$. 
\end{lemma}

\begin{proof}
Let $v$ be a solution to the Yamabe problem 
on $(M,g)$, i.\,e.\,a smooth positive function 
on $M$ such that the Riemannian metric 
$v^{4/(n-2)}g$ has constant scalar curvature $1$ 
on $M$. 
We choose a smooth function 
$\chi_{\alpha}$: $N_{\epsilon}\to[0,1]$ such that 
$\chi_{\alpha}\equiv 1$ on 
$N\setminus U^{N}_{\epsilon}(2\alpha)$ and 
$\chi_{\alpha}\equiv 0$ on 
$U^{N}_{\epsilon}(\alpha)$. 
Then for every $\theta$ the function 
\begin{displaymath}
v_{\theta}
:=F^{-\frac{n-2}{2}}v\chi_{\alpha}+1-\chi_{\alpha}
\end{displaymath}
on $N_{\epsilon}$ is smooth and positive 
and it depends on $\theta$ since 
$F$ depends on $\theta$. 
Now there exist constants 
$b_{\alpha},B_{\alpha}>0$ such that 
for every $\theta$ we have 
\begin{equation}
\label{v_theta_bound}
b_{\alpha}\leq v_{\theta}\leq B_{\alpha}
\end{equation}
on $N_{\epsilon}$. 
For every $\theta$ we define the Riemannian metric 
\begin{displaymath}
\tilde{g}_{\theta}
:=v_{\theta}^{\frac{4}{n-2}}g_{\theta}
\end{displaymath}
on $N_{\epsilon}$. 
Let $\theta$ be so small that on 
$N\setminus U^{N}_{\epsilon}(2\alpha)$ we have 
$g_{\theta}=F^2 g$. 
Then on $N\setminus U^{N}_{\epsilon}(2\alpha)$
we get $\tilde{g}_{\theta}=v^{4/(n-2)}g$ and 
thus 
\begin{equation}
\label{scal_g_theta_tilde}
\scal_{\tilde{g}_{\theta}}\equiv 1\textrm{ on }
N\setminus U^{N}_{\epsilon}(2\alpha).
\end{equation}
For every $\theta$ we define the function 
\begin{displaymath}
\tilde{u}_{\theta}=\frac{u_{\theta}}{v_{\theta}}
\end{displaymath}
on $N_{\epsilon}$ and we choose 
$x_{\theta}'\in N_{\epsilon}$ such that 
\begin{displaymath}
\tilde{u}_{\theta}(x_{\theta}')
=\max_{N_{\epsilon}}\tilde{u}_{\theta}. 
\end{displaymath}
Then for all $\theta$ we have 
by (\ref{v_theta_bound}) 
\begin{equation}
\label{u_theta_tilde_x_theta_prime}
\tilde{u}_{\theta}(x_{\theta}')
\geq\tilde{u}_{\theta}(x_{\theta})
=\frac{m_{\theta}}{v_{\theta}(x_{\theta})}
\geq\frac{m_{\theta}}{B_{\alpha}}
\end{equation}
and thus by our assumption 
$\tilde{u}_{\theta}(x_{\theta}')\to\infty$ 
as $\theta\to 0$. 
By the conformal transformation law 
(\ref{L_g_conform}) for $L_{g_{\theta}}$ 
we have at $x_{\theta}'$
\begin{equation}
\label{x_theta_prime}
(\Delta_{\tilde{g}_{\theta}}\tilde{u}_{\theta})
(x_{\theta}')
+\frac{n-2}{4(n-1)}
\scal_{\tilde{g}_{\theta}}(x_{\theta}')
\tilde{u}_{\theta}(x_{\theta}')
=-F_{\eta}(x_{\theta}')
v_{\theta}(x_{\theta}')^{-\frac{n+2}{n-2}}.
\end{equation}
Notice that the right hand side is bounded 
independently of $\theta$ since on $\supp(F_{\eta})$ 
the function $v_{\theta}$ is independent of $\theta$. 
Since the first term on the left hand side is 
non-negative and since 
$\tilde{u}_{\theta}(x_{\theta}')\to\infty$ 
as $\theta\to 0$ it follows that 
$\scal_{\tilde{g}_{\theta}}(x_{\theta}')\to 0$ 
as $\theta\to 0$. 
Thus by (\ref{scal_g_theta_tilde}) we have 
$x_{\theta}'\in U^{N}_{\epsilon}(2\alpha)$ 
if $\theta$ is small enough.
It remains to prove the inequalities of the 
assertion. 
First, by definition of $m_{\theta}$ we have 
$u_{\theta}(x_{\theta}')\leq m_{\theta}$. 
Second, by (\ref{v_theta_bound}) and 
(\ref{u_theta_tilde_x_theta_prime}) we have 
\begin{displaymath}
u_{\theta}(x_{\theta}')
=v_{\theta}(x_{\theta}')
\tilde{u}_{\theta}(x_{\theta}')
\geq 
\frac{v_{\theta}(x_{\theta}')m_{\theta}}{B_{\alpha}}
\geq 
\frac{b_{\alpha}}{B_{\alpha}}m_{\theta}.
\end{displaymath}
This finishes the proof of the lemma. 
\end{proof}

In the remaining part of the proof of Step 5 
we distinguish two cases.\\

\noindent 
\textbf{Case 1:} 
There exists $c>0$ such that 
$x_{\theta}'\in N\setminus U^{N}_{\epsilon}(c)$ 
for an infinite number of $\theta$.

The proof is very similar to Subcase I.1 
in the proof of Theorem 6.1 in 
\cite{ammann.dahl.humbert:13}. 
After taking a subsequence we may assume that 
there exists 
$\overline{x}\in N\setminus U^{N}_{\epsilon}(c)$ 
such that 
$\lim_{\theta\to 0}x_{\theta}'=\overline{x}$. 
For every $\theta$ we put 
$a_{\theta}:=u_{\theta}(x_{\theta}')$. 
In a neighborhood $U$ of $\overline{x}$ the 
metric $g_{\theta}=F^2g$ is independent of $\theta$ 
if $\theta$ is small enough. 
We define 
$\tilde{g}_{\theta}:=a_{\theta}^{4/(n-2)}g_{\theta}$.
Let $r>0$. 
We apply Lemma \ref{lemma_limit_space_1} 
with $O=U$, $\alpha=\theta$, 
$q_{\alpha}=x_{\theta}'$, 
$q=\overline{x}$, 
$\gamma_{\alpha}=g_{\theta}=F^2g$ 
and $b_{\alpha}=a_{\theta}^{2/(n-2)}$. 
For $\theta$ small we then obtain a diffeomorphism 
\begin{displaymath}
\Theta_{\theta}:\,B^n(r)\to 
B^{g_{\theta}}
(x_{\theta}',a_{\theta}^{-\frac{2}{n-2}}r)
\end{displaymath}
such that the sequence of metrics 
$(\Theta_{\theta}^*(\tilde{g}_{\theta}))$ converges 
to the flat metric $\xi^n$ in $C^2(B^n(r))$. 
For all sufficiently small $\theta$ we have 
\begin{displaymath}
B^{g_{\theta}}
(x_{\theta}',a_{\theta}^{-\frac{2}{n-2}}r)
\cap\supp(F_{\eta})=\emptyset 
\end{displaymath}
and thus $L_{g_{\theta}}u_{\theta}=0$ on 
$B^{g_{\theta}}(x_{\theta}',a_{\theta}^{-2/(n-2)}r)$. 
We define 
$\tilde{u}_{\theta}:=a_{\theta}^{-1}u_{\theta}$. 
By the conformal transformation law 
\ref{L_g_conform} for $L_{g_{\theta}}$ we have 
\begin{displaymath}
L_{\tilde{g}_{\theta}}\tilde{u}_{\theta}=0
\end{displaymath}
on 
$B^{g_{\theta}}(x_{\theta}',a_{\theta}^{-2/(n-2)}r)$ 
and since 
$dv^{\tilde{g}_{\theta}}=a_{\theta}^p dv^{g_{\theta}}$ 
we have 
\begin{align*}
\int_{B^{g_{\theta}}
(x_{\theta}',a_{\theta}^{-\frac{2}{n-2}}r)}
\tilde{u}_{\theta}^p\,dv^{\tilde{g}_{\theta}}
&=\int_{B^{g_{\theta}}
(x_{\theta}',a_{\theta}^{-\frac{2}{n-2}}r)}
u_{\theta}^p\,dv^{g_{\theta}}\\
&\leq\int_{N}u_{\theta}^p\,dv^{g_{\theta}}\\
&\leq C_1
\end{align*}
by Step 4. 
Since 
\begin{displaymath}
\Theta_{\theta}:\,
(B^n(r),\Theta_{\theta}^*(\tilde{g}_{\theta}))\to 
(B^{g_{\theta}}
(x_{\theta}',a_{\theta}^{-\frac{2}{n-2}}r),
\tilde{g}_{\theta})
\end{displaymath}
is an isometry we can consider $\tilde{u}_{\theta}$ 
as a solution of 
\begin{displaymath}
L_{\Theta_{\theta}^*(\tilde{g}_{\theta})}
\tilde{u}_{\theta}=0
\end{displaymath}
on $B^n(r)$ with 
\begin{displaymath}
\int_{B^n(r)} \tilde{u}_{\theta}^p\,
dv^{\Theta_{\theta}^*(\tilde{g}_{\theta})}\leq C_1.
\end{displaymath}
Since 
$\|\tilde{u}_{\theta}\|_{L^{\infty}(B^n(r))}
=|\tilde{u}_{\theta}(0)|=1$ 
we can apply Lemma \ref{lemma_limit_space_2} 
with $V=\mR^n$, $\alpha=\theta$, 
$g_{\alpha}=\Theta_{\theta}^*(\tilde{g}_{\theta})$ 
and $u_{\alpha}=\tilde{u}_{\theta}$. 
We can apply this lemma since every compact subset 
of $\mR^n$ is contained in some ball $B^n(r)$. 
We conclude that there exists a non-negative 
$C^2$-function $u$ on $\mR^n$ such that 
\begin{displaymath}
L_{\xi^n}u=0,\quad u(0)=1,
\end{displaymath}
in particular $u\not\equiv 0$. 
By (\ref{limit_L_r_norm}) we have for every $r>0$
\begin{displaymath}
\int_{B^n(r)}u^p\,dv^{\xi^n}
=\lim_{\theta\to 0}
\int_{B^{g_{\theta}}(x_{\theta}',a_{\theta}
^{-\frac{2}{n-2}}r)} u_{\theta}^p\,dv^{g_{\theta}}
\leq C_1.
\end{displaymath}
In particular 
\begin{displaymath}
\int_{\mR^n}u^p\,dv^{\xi^n}\leq C_1.
\end{displaymath}
After dividing $u$ by a constant we may assume 
that $\int_{\mR^n}u^pdv^{\xi^n}=1$. 
We have obtained a contradiction to Lemma 
\ref{lemma_limit_space_3}. 
This finishes the proof in Case 1.\\

\noindent 
\textbf{Case 2:} 
For every $c>0$ we have 
$x_{\theta}'\in U^{N}_{\epsilon}(c)$ for $\theta$ 
sufficiently small.

The proof is very similar to Subcase I.2 in the proof 
of Theorem 6.1 in \cite{ammann.dahl.humbert:13}. 
Again for every $\theta$ we put 
$a_{\theta}:=u_{\theta}(x_{\theta}')$. 
The subset $U^{N}_{\epsilon}(c)$ is diffeomorphic 
to $W\times I\times S^{n-k-1}$ where $I$ is an 
interval. 
We identify 
\begin{displaymath}
x_{\theta}'=(y_{\theta},t_{\theta},z_{\theta})
\in W\times 
(-\ln R_0+\ln\epsilon,-\ln\epsilon+\ln R_0)
\times S^{n-k-1}.
\end{displaymath}
By taking a subsequence we may assume that 
$y_{\theta}$, $\frac{t_{\theta}}{A_{\theta}}$ 
and $z_{\theta}$ converge respectively to 
$y\in W$, $T\in[-\infty,\infty]$ and 
$z\in S^{n-k-1}$. 
First we apply Lemma \ref{lemma_limit_space_1} 
with $V=W$, $\alpha=\theta$, 
$q_{\alpha}=y_{\theta}$, 
$q=y$, $\gamma_{\alpha}=\tilde{h}_{t_{\theta}}$, 
$\gamma_0=\tilde{h}_T$ 
and $b_{\alpha}=a_{\theta}^{2/(n-2)}$, where we 
define $\tilde{h}_{-\infty}=h_1$ 
and $\tilde{h}_{\infty}=h_2$. 
For every $r>0$ the lemma provides diffeomorphisms 
\begin{displaymath}
\Theta_{\theta}^{y}:\,
B^k(r)\to 
B^{\tilde{h}_{t_{\theta}}}(y_{\theta},
a_{\theta}^{-\frac{2}{n-2}}e^{-f(t_{\theta})}r)
\end{displaymath}
such that 
$(\Theta_{\theta}^{y})^*
(a_{\theta}^{4/(n-2)}e^{2f(t_{\theta})}
\tilde{h}_{t_{\theta}})$ 
converges to the flat metric $\xi^k$ on $B^k(r)$ 
as $\theta\to 0$. 
Second we apply Lemma \ref{lemma_limit_space_1} 
with $V=S^{n-k-1}$, $\alpha=\theta$, 
$q_{\alpha}=z_{\theta}$, 
$\gamma_{\alpha}=\gamma_0=\sigma^{n-k-1}$ 
and $b_{\alpha}=a_{\theta}^{2/(n-2)}$. 
For every $r'>0$ we obtain diffeomorphisms 
\begin{displaymath}
\Theta_{\theta}^{z}:\,
B^{n-k-1}(r')\to B^{\sigma^{n-k-1}}
(z_{\theta},a_{\theta}^{-\frac{2}{n-2}}r')
\end{displaymath}
such that 
$(\Theta_{\theta}^{z})^*
(a_{\theta}^{4/(n-2)}\sigma^{n-k-1})$ converges 
to $\xi^{n-k-1}$ on $B^{n-k-1}(r')$ as 
$\theta\to 0$. 
For $r,r',r''>0$ we define 
\begin{align*}
U_{\theta}(r,r',r'')
&:=
B^{\tilde{h}_{t_{\theta}}}
(y_{\theta},a_{\theta}^{-\frac{2}{n-2}}
e^{-f(t_{\theta})}r)\times 
[t_{\theta}-a_{\theta}^{-\frac{2}{n-2}}r'',
t_{\theta}+a_{\theta}^{-\frac{2}{n-2}}r'']\\
&\quad \times B^{\sigma^{n-k-1}}
(z_{\theta},a_{\theta}^{-\frac{2}{n-2}}r')
\end{align*}
and 
\begin{displaymath}
\Theta_{\theta}:\,
B^k(r)\times[-r'',r'']\times B^{n-k-1}(r')\to 
U_{\theta}(r,r',r'')
\end{displaymath}
by 
\begin{displaymath}
\Theta_{\theta}(y,s,z):=
(\Theta_{\theta}^{y}(y),t(s),
\Theta_{\theta}^{z}(z)),
\end{displaymath}
where $t(s):=t_{\theta}+a_{\theta}^{-2/(n-2)}s$. 
Then $\Theta_{\theta}$ is a diffeomorphism 
and we obtain 
\begin{displaymath}
\Theta_{\theta}^*
(a_{\theta}^{\frac{4}{n-2}}g_{\theta})=
(\Theta_{\theta}^{y})^*
(a_{\theta}^{\frac{4}{n-2}}e^{2f(t)}\tilde{h}_t)
+ds^2
+(\Theta_{\theta}^{z})^*
(a_{\theta}^{\frac{4}{n-2}}\sigma^{n-k-1})
+\Theta_{\theta}^*
(a_{\theta}^{\frac{4}{n-2}}\widetilde{T}_t).
\end{displaymath}
As in Subcase I.2 in the proof of Theorem 6.1 
in \cite{ammann.dahl.humbert:13} one shows that 
the sequence of Riemannian metrics 
$\Theta_{\theta}^*(a_{\theta}^{4/(n-2)}g_{\theta})$ 
tends to the flat metric $\xi^n$. 
Then as in the proof of Case 1 above one obtains 
a non-negative $C^2$-function $u$ 
satisfying 
\begin{displaymath}
L_{\xi^n}u=0,\quad u(0)=1,\quad 
\int_{\mR^n}u^p\,dv^{\xi^n}<\infty.
\end{displaymath}
In particular $u\not\equiv 0$ 
and one obtains a contradiction to Lemma 
\ref{lemma_limit_space_3} as above. 
This finishes the proof of Step 5.\\

By Steps 3 and 5 
we know that there exist $a>0$ and $C>0$ such that 
for every $\theta$ we have 
\begin{equation}
\label{u_theta_L2_bound}
\int_{U^{M}(a)}u_{\theta}^2\,dv^{g_{\theta}}
\leq C.
\end{equation}
We recall that for $\alpha>0$ we have defined 
\begin{displaymath}
A_{\alpha}:=U^{M}(2\alpha)\setminus U^{M}(\alpha)
\subset M.
\end{displaymath}
Next we define 
\begin{displaymath}
E:=\liminf_{\alpha\to 0}\liminf_{\theta\to 0}
\int_{A_{\alpha}} u_{\theta}^2\,dv^{g_{\theta}}.
\end{displaymath}

\noindent 
\textbf{Step 6:} 
Conclusion.

By the result of Step 2 it remains to show 
that $E=0$. 
We proceed similarly as on p.\,50 of the article 
\cite{ammann.dahl.humbert:13}. 
Namely there exists $\delta>0$ such that for every 
$\alpha\in(0,\delta)$ we have 
\begin{displaymath}
\liminf_{\theta\to 0}
\int_{A_{\alpha}} u_{\theta}^2\,dv^{g_{\theta}}
\geq \frac{E}{2}.
\end{displaymath}
For $m\in\mN$ we set $\alpha_m:=2^{-m}\delta$. 
Then we have 
\begin{displaymath}
\liminf_{\theta\to 0}
\int_{A_{\alpha_m}} u_{\theta}^2\,dv^{g_{\theta}}
\geq \frac{E}{2}
\end{displaymath}
for all $m$. 
Let $N_0\in\mN$. 
The sets $A_{\alpha_m}$, $m\in\mN$, are disjoint and 
therefore we have 
\begin{displaymath}
\int_{U^{M}(\delta)}u_{\theta}^2\,dv^{g_{\theta}}
\geq 
\int_{\bigcup_{m=1}^{N_0}A_{\alpha_m}}
u_{\theta}^2\,dv^{g_{\theta}}
=\sum_{m=1}^{N_0}
\int_{A_{\alpha_m}} u_{\theta}^2\,dv^{g_{\theta}}
\end{displaymath}
for all $\theta$. 
From this we obtain 
\begin{align*}
\liminf_{\theta\to 0}
\int_{U^{M}(\delta)}
u_{\theta}^2\,dv^{g_{\theta}}
&\geq 
\liminf_{\theta\to 0}
\sum_{m=1}^{N_0}
\int_{A_{\alpha_m}} 
u_{\theta}^2\,dv^{g_{\theta}}\\
&\geq 
\sum_{m=1}^{N_0}\liminf_{\theta\to 0}
\int_{A_{\alpha_m}} 
u_{\theta}^2\,dv^{g_{\theta}}\\
&\geq 
\frac{E N_0}{2}.
\end{align*}
Assume that $E>0$. 
Since $N_0\in\mN$ can be chosen arbitrarily large, 
we obtain a contradiction to the estimate 
(\ref{u_theta_L2_bound}). 
Thus we have $E=0$ and Theorem 
\ref{theorem_preservation_mass} is proved.

\section{Application to the positive mass conjecture}
\label{section_pmt_non_spin}

In this section we study an application of 
Theorem \ref{theorem_preservation_mass} to the 
positive mass conjecture. 
By a simply connected manifold $T$ we mean a connected 
manifold $T$ with $\pi_1(T)=\{0\}$. 
If $T$ is an oriented manifold, we denote by $-T$ 
the manifold $T$ with the opposite orientation. 

\begin{lemma}
\label{lemma_orient_cobord}
Let $X_1$ be a closed 
simply connected oriented non-spin manifold 
of dimension $n\geq 5$ and let $X_0$ be a manifold 
of dimension $n$ which is oriented cobordant to 
$X_1$. 
Then $X_1$ can be obtained from $X_0$ by 
finitely many surgeries of dimension 
$k\in\{0,...,n-3\}$. 
\end{lemma}

\begin{proof}
The assertion follows from the proof of Theorem C 
in the article \cite{gromov.lawson:80} 
by Gromov and Lawson. 
Namely let $W$ be an oriented cobordism from 
$X_0$ to $X_1$. 
After applying finitely many surgeries of dimension 
$0$ or $1$ to $X_0$ and then to $W$ we may assume 
that $X_0$ and $W$ are simply connected. 
After further applying surgeries and using that 
$X_1$ is not spin we can assume that the 
induced homomorphism $\pi_2(X_1)\to\pi_2(W)$ is 
surjective. 
It follows that for $i\leq 2$ we have 
$H_i(W,X_0)=0$ and $H_i(W,X_1)=0$. 
The assertion then follows from a result by Smale 
(\cite{smale:62}, see also 
\cite[VIII Thm. 4.1]{kosinski:93}). 
\end{proof}

\begin{definition}
We say that a closed manifold $M$ satisfies PMT 
if for every Riemannian metric $g$ on $M$ 
with $Y(M,g)>0$ and for every point $p\in M$ 
such that $g$ is flat on an open neighborhood 
of $p$ we have $m(M,g)\geq 0$ at $p$.
\end{definition}

\begin{lemma}
\label{pmt_a}
Let $M$, $N$ be two closed manifolds of dimension 
$n$ such that $N$ satisfies PMT. 
Assume that $M$ is obtained from $N$ by surgery of 
dimension $\ell\in\{2,...,n-1\}$. 
Then $M$ satisfies PMT. 
\end{lemma}

\begin{proof}
In general any surgery of dimension $\ell$ on a 
manifold of dimension $n$ can be undone by a 
surgery of dimension $n-1-\ell$. 
Thus $N$ can be obtained from $M$ by surgery of 
dimension $k\in\{0,...,n-3\}$ and the assertion 
follows from Theorem 
\ref{theorem_preservation_mass}. 
\end{proof}

\begin{lemma}
\label{pmt_conn_sum}
Let $M$ and $P$ be two closed manifolds of the 
same dimension. 
Assume that $M$ does not satisfy PMT and that 
there exists a Riemannian metric $h$ on $P$ 
with $Y(P,h)>0$. 
Then the connected sum $M\# P$ does not satisfy 
PMT. 
\end{lemma}

\begin{proof}
Let $g$ be a Riemannian metric on $M$ with 
$Y(M,g)>0$ such that 
at some point $p\in M$ we have $m(M,g)<0$. 
The metric $g\amalg h$ on the disjoint union 
$M\amalg P$ satisfies 
\begin{displaymath}
Y(M\amalg P,g\amalg h)
=\min\{Y(M,g), Y(P,h)\}>0
\end{displaymath}
(see e.\,g.Section 1.2 in 
\cite{ammann.dahl.humbert:13}). 
The Green function of $L_{g\amalg h}$ is given by 
\begin{displaymath}
G_{g\amalg h}=\left\{
\begin{array}{ll}
G_g & \textrm{on }M,\\ 
0   & \textrm{on }P
\end{array} 
\right.
\end{displaymath}
and thus at $p$ we have 
$m(M\amalg P,g\amalg h)=m(M,g)<0$, 
i.\,e.\,$M\amalg P$ does not satisfy PMT. 
Since $M\# P$ can be obtained from $M\amalg P$ by 
surgery of dimension~$0$ Theorem 
\ref{theorem_preservation_mass} shows that 
$M\# P$ does not satisfy PMT. 
\end{proof}

\begin{theorem}
\label{theorem_pmt}
Assume that there exists a closed orientable 
simply-connected non-spin 
manifold of dimension $n\geq 5$ satisfying PMT. 
Then every closed manifold of dimension $n$ 
satisfies PMT. 
\end{theorem}

\noindent Note that, by Proposition 4.1 in \cite{schoen:89} or Section 5 in \cite{lohkamp:99}, this theorem could also be stated for the ADM-mass in the context of the standard positive mass conjecture coming from general relativity. 

\begin{proof}
Let $M$ be a closed oriented simply-connected non-spin 
manifold 
of dimension $n$ satisfying PMT. 
The manifold $M\#M\#(-M)$ is oriented cobordant to 
$M$. 
By Lemma \ref{lemma_orient_cobord} the manifold 
$M$ can be obtained from $M\#M\#(-M)$ by finitely 
many surgeries of dimension $k\in\{0,...,n-3\}$. 
Therefore $M\#M\#(-M)$ can be obtained from $M$ 
by finitely many surgeries of dimension 
$\ell\in\{2,...,n-1\}$. 
Since $M$ satisfies PMT it follows from 
Lemma \ref{pmt_a} that 
$M\#M\#(-M)$ satisfies PMT. 
By Lemma \ref{pmt_conn_sum} we conclude that 
$M\#(-M)$ satisfies PMT. 

Let $N$ be a closed manifold of dimension $n$. 
Assume first that $N$ is orientable and choose 
an orientation on $N$. 
Assume that $N$ does not satisfy PMT. 
By Lemma \ref{pmt_conn_sum} it follows that 
$N\#(-N)$ does not satisfy PMT. 
Now $N\#(-N)$ is oriented cobordant to $M\#(-M)$ 
since both manifolds are oriented cobordant to $S^n$. 
Furthermore $M\#(-M)$ is simply connected and 
non-spin. 
By Lemma \ref{lemma_orient_cobord} the manifold 
$M\#(-M)$ can be obtained from $N\#(-N)$ by finitely 
many surgeries of dimension $k\in\{0,...,n-3\}$. 
By Theorem \ref{theorem_preservation_mass} the 
manifold $M\#(-M)$ does not satisfy PMT which is 
a contradiction. 

Next assume that $N$ is not orientable. 
Let $\pi$: $\tilde{N}\to N$ be the two-fold 
orientable covering of $N$. 
Let $g$ be a Riemannian metric on $N$ which is flat 
on an open neighborhood of a point $p\in N$ 
and such that $L_g$ is a positive operator. 
Let $\tilde{g}$ be the Riemannian metric on 
$\tilde{N}$ such that $\pi$ is a Riemannian 
covering. 
Since the first eigenvalue $\tilde{\lambda}_0$ of 
$L_{\tilde{g}}$ is simple and the corresponding 
eigenfunctions do not change their sign, 
$\tilde{\lambda}_0$ is also an eigenvalue of 
$L_g$. 
It follows that $L_{\tilde{g}}$ is a positive 
operator. 
Now if we write 
$\pi^{-1}(p)=\{\tilde{p}_1,\tilde{p}_2\}$ 
and if $\tilde{G}_1$, $\tilde{G}_2$ denote 
the Green functions for $L_{\tilde{g}}$ 
at $\tilde{p}_1$ and $\tilde{p}_2$ respectively, 
then for the Green function $G$ of $L_g$ at $p$ 
we have $G\circ\pi=\tilde{G}_1+\tilde{G}_2$. 
In particular if 
$m^{\tilde{p}_1}(\tilde{N},\tilde{g})$ denotes 
the mass of $(\tilde{N},\tilde{g})$ 
at $\tilde{p}_1$, then for the mass of $L_g$ 
at $p$ we have 
$m(N,g)=m^{\tilde{p}_1}(\tilde{N},\tilde{g})
+\tilde{G}_2(\tilde{p}_1)>0$. 
\end{proof}

\noindent  It is easy to find examples of closed 
orientable simply-connected non-spin manifolds, 
e.\,g.\,$\mC P^{2m}$ or $\mC P^{2m}\times S^k$ 
with $k\geq 2$. 
Our hope is that among these examples one can 
find manifolds of dimension at least $8$ 
satisfying PMT. 
However we have not yet succeeded. 
Among the manifolds of dimension at least 
$8$ satisfying PMT we know examples 
which are simply connected and spin 
(by Section \ref{section_pmt_spin}) 
and examples 
which are not simply-connected and non-spin: indeed, we have

\begin{proposition}
 Let $n \geq 5$, $n\equiv 1 \bmod 4$. Then, the projective space $\mR P^n$ satisfies PMT. 
\end{proposition}

\begin{proof}
Let $g$ be a metric on $\mR P^n$ which is flat around $p \in \mR P^n$ such that $L_g$ is a positive 
operator. 
Using the two-fold covering $S^n\to\mR P^n$ one 
obtains as in the last part of the proof of 
Theorem \ref{theorem_pmt} that the mass 
of~$L_g$ at~$p$ is strictly positive. 
\end{proof}

\end{document}